\documentclass[]{scrartcl}
\usepackage{amsmath,amsthm,amssymb,thmtools,aligned-overset,amsfonts,amscd,mathrsfs,yhmath}
\usepackage{aurical}
\usepackage{authblk,orcidlink}
\usepackage{xcolor,graphicx}
\usepackage{xstring,bbding}
\usepackage{enumerate,enumitem}
\setenumerate[1]{label=(\arabic*)}	
\setenumerate[2]{label=(\alph*)}
\setenumerate[3]{label=(\roman*)}
\usepackage{mleftright}
\usepackage[left=3cm,bottom=3cm,top=3cm]{geometry}
\usepackage{changepage}
\usepackage{hyperref}
\usepackage{cleveref}
\usepackage{biblatex}
\DeclareMathAlphabet\mathzapf{T1}{pzc} {m} {it}
\newtheorem{Satz}{Theorem}[section]
\newtheorem{Lemma}[Satz]{Lemma}
\newtheorem{Korollar}[Satz]{Corollary}
\newtheorem{Proposition}[Satz]{Proposition}
\newtheorem{Definition}[Satz]{Definition}
\newtheorem{Beispiel}[Satz]{Example}
\newtheorem{Bemerkung}[Satz]{Remark}
\crefname{Beispiel}{example}{examples}
\Crefname{Beispiel}{Example}{Examples}

\newcommand{\defi}[1]{\textbf{#1}}
\newcommand{\Stapel}[1]{
	\vcenter{%
		\Let@ \restore@math@cr \default@tag
		\baselineskip\fontdimen10 \scriptfont\tw@
		\advance\baselineskip\fontdimen12 \scriptfont\tw@
		\lineskip\thr@@\fontdimen8 \scriptfont\thr@@
		\lineskiplimit\lineskip
		\halign{\hfil$\m@th\displaystyle##$\hfil&$\m@th\displaystyle{}##$\hfil\crcr#1\crcr}%
	}%
}
\newcommand{\Menge}[3][]{					
	\mleft\{\,\begin{aligned}#2\end{aligned}\:\middle|\:\begin{aligned}#3\end{aligned}\,\mright\}
}

\DeclareMathOperator{\abschluss}{cl}			
\NewDocumentCommand{\cl}{om}{\ensuremath{\IfNoValueTF{#1}{\abschluss}{\abschluss}\IfNoValueF{#2}{\left(#2\right)}}}
\DeclareMathOperator*{\argmax}{argmax}			
\DeclareMathOperator*{\argmin}{argmin}			
\newcommand{\argminKegel}[1]{\sideset{}{_{#1}} \argmin} 
\DeclareMathOperator{\balancedcore}{balcore}
\NewDocumentCommand{\balancedCore}{om}{\IfNoValueTF{#1}{\balancedcore\left(#2\right)}{\balancedcore_{#1}\left(#2\right)}}
\DeclareMathOperator{\balancedhull}{bal}
\NewDocumentCommand{\balancedHull}{om}{\IfNoValueTF{#1}{\balancedhull\left(#2\right)}{\balancedhull_{#1}\left(#2\right)}}
\DeclareMathOperator{\unitBall}{B}
\NewDocumentCommand{\ball}{omm}{\unitBall_{#1}\left(#2;#3\right)} 
\NewDocumentCommand{\ballclosed}{omm}{\overline\unitBall_{#1}\left(#2;#3\right)} 

\DeclareMathOperator{\bd}{bd}					
\NewDocumentCommand{\boundary}{om}{\IfNoValueTF{#1}{\bd\left(#2\right)}{\bd\left(#2\right)}}
\NewDocumentCommand{\convergentSequencesSpace}{o}{\IfNoValueTF{#1}{c}{c\left(#1\right)}}
\DeclareMathOperator{\core}{core}				
\NewDocumentCommand{\aint}{om}{\IfNoValueTF{#1}{\ensuremath{\core\left(#2\right)}}{\ensuremath{\core_{#1}\left(#2\right)}}}
\NewDocumentCommand{\HenigDilatingCones}{m}{\mathop{\mathrm{DC}}\left(#1\right)}
\NewDocumentCommand{\HenigDilatingConesVB}{m}{\operatorname{DC}^\vb\left(#1\right)}
\NewDocumentCommand{\HenigDilatingConesVBIntersection}{omm}{\operatorname{DC}^\vb_{#2\IfNoValueF{#1}{,#1}}\left(#3\right)}

\DeclareMathOperator{\dualkegelalg}{a-dc}		
\NewDocumentCommand{\dualconealg}{om}{\IfNoValueTF{#1}{\dualkegelalg\left(#2\right)}{\dualkegelalg_{#1}\left(#2\right)}}
\DeclareMathOperator{\dualkegeltop}{t-dc}		
\NewDocumentCommand{\dualconetop}{om}{\IfNoValueTF{#1}{\dualkegeltop\left(#2\right)}{\dualkegeltop_{#1}\left(#2\right)}}

\DeclareMathOperator{\effizienteMenge}{Eff}		
\NewDocumentCommand{\Eff}{mmm}{\effizienteMenge\left(#1,#2,#3\right)}
\DeclareMathOperator{\icore}{icore}				
\NewDocumentCommand{\raint}{om}{\IfNoValueTF{#1}{\ensuremath{\icore\left(#2\right)}}{\ensuremath{\icore_{#1}\left(#2\right)}}}
\DeclareMathOperator{\inneres}{int}				
\NewDocumentCommand{\interior}{om}{\ensuremath{\IfNoValueTF{#1}{\inneres}{\inneres}\left(#2\right)}}

\newcommand{\leereMenge}{\emptyset}				
\DeclareMathOperator{\lev}{lev}					
\newcommand{\lin}{\mathop{\mathrm{lin}}}		
\newcommand{\ohne}{\setminus}					
\DeclareMathOperator{\Peffhenig}{PEff}
\NewDocumentCommand{\PEffHenig}{mmm}{\Peffhenig\left(#1,#2,#3\right)}

\DeclareMathOperator{\Projektion}{Proj}			

\NewDocumentCommand{\ProjectionSet}{mm}{\ensuremath{\Projektion_{#1}\left(#2\right)}}
\DeclareMathOperator{\relativeinterior}{relint}	
\NewDocumentCommand{\relint}{om}{\ensuremath{\IfNoValueTF{#1}{\relativeinterior}{\relativeinterior_{#1}}\left(#2\right)}}

\NewDocumentCommand{\SetRestrictionToPoint}{mm}{\left.#1\right|_{#2}}
\DeclareMathOperator{\weaklyEff}{WEff}			
\NewDocumentCommand{\WEff}{mmm}{\weaklyEff\left(#1,#2,#3\right)}
\NewDocumentCommand{\zeroSequencesSpace}{o}{\IfNoValueTF{#1}{c_0}{c_0\left(#1\right)}}

\DeclareRobustCommand{\norm}[2][]{			
	\begingroup
	\protect
	\if\relax\detokenize{#1}\relax			
	\ensuremath{\left\lVert#2\right\rVert}
	\else									
	\IfSubStr{bigBigbiggBigg}{#1}{
		\ensuremath{\#1\lVert#2\#1\rVert}
	}{
		\PackageWarning{Diese Größe ist mir nicht bekannt.}
	}
	\fi
	\endgroup
}
\newcommand{\skp}[3][]{						
	\ensuremath{\mleft\langle\,#2\:\middle|\:#3\,\mright\rangle}
}
\DeclareRobustCommand{\qed}{%
	\ifmmode \mathqed
	\else
	\leavevmode\unskip\penalty9999 \hbox{}\nobreak\protect\hfill\quad\hbox{\text{$_\blacksquare$}}%
	\fi
}

\newcommand\mathcalx{\text{\Fontskrivan \slshape x}}
\newcommand{\x}{\mathcalx}
\newcommand{\calS}{\mathcal S}
\newcommand{\s}{\calS}

\renewcommand{\c}{C}
\newcommand{\calH}{\mathcal H}
\newcommand{\calD}{\mathcal D}
\newcommand{\calU}{\mathcal U}

\newcommand{\T}{\mathbb T}

\newcommand{\N}{\mathbb N}
\newcommand{\R}{\mathbb R}
\newcommand{\B}{\mathbb B}

\newcommand{\vb}{\otimes}

\newcommand{\fvb}{f^\vb}

\newcommand{\meinUrbildraum}{X}
\newcommand{\meinBildraum}{Y}

\begin{document}
    
    \title{Proper efficiency results in vector optimisation in real linear-topological spaces based on vectorial penalisation}

    \author[1]{{Christian} {Günther} \orcidlink{0000-0002-1491-4896}
    \thanks{\href{mailto:c.guenther@ifam.uni-hannover.de}{c.guenther@ifam.uni-hannover.de}}}
    \author[2]{{Elisabeth} {Köbis}
    \thanks{\href{mailto:elisabeth.kobis@ntnu.no}{elisabeth.kobis@ntnu.no}}}
    \author[2]{{Paul} {Schmölling} \orcidlink{0009-0008-4157-0903}
    \thanks{\href{mailto:paul.schmolling@ntnu.no}{paul.schmolling@ntnu.no}}}
    \author[3]{{Christiane} {Tammer}
    \thanks{\href{mailto:christiane.tammer@mathematik.uni-halle.de}{christiane.tammer@mathematik.uni-halle.de}}}
    
    \affil[1]{{Institut für Angewandte Mathematik}, {Leibniz Universität Hannover}, {{Welfengarten 1}, {Hannover}, {30167},  {Germany}}}
    \affil[2]{{Department of Mathematics}, {Norwegian University of Science and Technology}, {{Alfred Getz' vei 1}, {Trondheim}, {7034}, {Norway}}}
    \affil[3]{{Faculty of Natural Sciences II, Institute of Mathematics}, {Martin-Luther-University Halle-Wittenberg}, {{Theodor-Lieser-Straße 5}, {Halle (Saale)}, {06120}, {Germany}}}

    \maketitle

    \begin{abstract}
        In this paper, we are dealing with constrained vector optimisation problems where the objective function acts between real linear-topological spaces. Our aim is to study the relationships between the sets of properly efficient solutions to constrained and unconstrained vector optimisation problems under certain cone convexity assumptions on the objective function using a vectorial penalisation approach.
    \end{abstract}

    \par\vskip\baselineskip\noindent

    \par\vskip\baselineskip\noindent

    \newpage
	\section{Introduction}
	In this paper, we study constrained vector optimisation problems where the objective function acts between two linear-topological spaces. Optimisation in general spaces provides a suitable framework for approximation theory, optimal control and optimisation under uncertainty; see \cite{GopfertRiahiTammerZalinescu2023VariationalMethodsPartiallyOrderedSpacesbook}, \cite{Jahn2011VectorOptimizationTheoryApplicationsExtensionsbook}, \cite{LeugeringSchiel2012RegularizednonlinearscalarizationvectoroptimizationproblemsPDEconstraintsjournalArticle} and \cite{Phu2001RoughConvergenceNormedLinearSpacesjournalArticle,Phu2003RoughConvergenceInfiniteDimensionalNormedSpacesjournalArticle}. The aim of our paper is to derive inclusions into the set of properly efficient solutions of the constrained vector optimisation problem by means of the sets of properly efficient solutions of two unconstrained vector optimisation problems and to identify sufficient conditions for exactness. To this end, we use a vectorial penalisation approach under suitable cone-convexity assumptions on the objective function. In scalar optimisation, at least two penalisation methods are commonly used, namely \textit{Infinite penalisation}, where one deals with extended real-valued functions, and \textit{Clarke penalisation}, where the distance to the feasible set is used as the penalisation term for locally Lipschitz objective functions. For details of these methods, see \cite{Ye2012ExactPenaltyPrinciplejournalArticle} and the references therein. Penalisation methods for scalar and vector optimisation problems are also discussed in \cite{ApetriiDureaStrugariu2014newpenalizationtoolscalarvectoroptimizationsjournalArticle}. Vectorial penalisation methods have been studied in particular in \cite{DureaStrugariuTammer2017MethodsDeriveNecessarySufficientOptimalityConditionsVectorOptimizationjournalArticle,GuntherTammerYao2018NecessaryoptimalityconditionsgeneralizedconvexmultiobjectiveoptimizationinvolvingnonconvexconstraintsjournalArticle} in order to derive useful necessary optimality conditions in vector optimisation with a comparatively simple structure, which in turn lead to effective algorithms.
    
    In our vectorial penalisation approach, the original constrained vector optimisation problem is replaced by two unconstrained problems to obtain a full characterisation of the constrained problem's solution set. In \cite{GuntherTammer2016RelationshipsconstrainedunconstrainedmultiobjectiveoptimizationapplicationlocationtheoryjournalArticle,GuntherTammer2018generalizedconvexconstrainedmultiobjectiveoptimizationjournalArticle,GuntherTammerYao2018NecessaryoptimalityconditionsgeneralizedconvexmultiobjectiveoptimizationinvolvingnonconvexconstraintsjournalArticle}, corresponding characterisations of the set of (weakly, strictly) efficient solutions of a constrained vector optimisation problem with a finite-dimensional image space of the objective function are derived by means of the vectorial penalisation technique. Moreover, \cite{GuntherKobisSchmollingTammer2023VectorialpenalisationvectoroptimisationreallineartopologicalspacesjournalArticle} studies a vectorial penalisation approach for (weakly, strictly) efficient elements of vector optimisation problems with a linear-topological image space equipped with a nontrivial, pointed, convex cone.
    
    In this paper, we focus on Henig-properly efficient elements of the vector optimisation problem. The advantage of using properly efficient elements is that they form a structured subclass of the set of (weakly) efficient elements having a specific geometric regularity, and therefore allow the preferences of the decision maker to be taken into account more precisely. Furthermore, properly efficient elements can be characterised by strictly monotone scalarisation functionals, which is important both in theoretical arguments, for instance in duality statements, and in numerical methods. Different concepts of properly efficient elements and their relationships are studied in \cite[Section~2.4]{KhanTammerZalinescu2015SetvaluedOptimizationIntroductionApplicationsbook}; see also the references therein. In this paper, we consider the concept of Henig-proper efficiency, where the Henig dilating cone is involved in the definition (see \cite{Henig1982ProperefficiencyrespectconesjournalArticlea}).

	Our paper is organised as follows. In \Cref{s-preml}, we introduce the constrained vector optimisation problem and state the standing assumptions. We then review the algebraic and topological properties of linear-topological spaces needed in our analysis and introduce the concept of cone convexity used in the main results. In \Cref{s-proper eff}, we discuss several solution concepts for vector optimisation problems. The main results are presented in \Cref{sec:vectorialPenalisation:PEff}. Under suitable cone-convexity assumptions on the objective function, we use vectorial penalisation to derive inclusions relating the Henig-properly efficient solution set of the constrained problem to the corresponding solution sets of two associated unconstrained problems. We then identify sufficient conditions under which these inclusions yield an exact characterisation. We also establish properties of the penalisation function. Finally, we summarise our conclusions in \Cref{s-con}.
    
	

	

	\section{Preliminaries}\label{s-preml}

	\subsection{Linear-topological spaces / Basic notations and notions}

    We briefly introduce some notations, mention basic definitions, and collect helpful statements for later results.

    In a linear space $X$ we denote the Minkowski sum of two sets $A,B\subseteq X$ by $A+B$ and if one of these sets is a singleton, $B=\{b\}$, we will simply write $A+b$.
    If $\T$ is a topological space and $A\subseteq \T$ we denote the interior of $A$ by $\interior{A}$, the closure of $A$ by $\cl{A}$ and the boundary of $A$ by $\boundary{A}$. For two given topological spaces $\T_1$, $\T_2$ we always consider the product space $\T_1\times\T_2$ to be equipped with the product topology.
    In a normed space $X$ we denote the closed ball around a point $x\in X$ with radius $\epsilon\geq0$ by $\B_X(x,\epsilon)$ and the zero of that space by $0_X$.
    In the specific case of the Euclidean space $\R^n$ we denote the standard ordering cone by $\R^n_\geq\coloneqq[0,\infty)^n$ and define $\R^n_\gneq\coloneqq\R^n_\geq\setminus\{0_{\R^n}\}$.

    Throughout, a cone in a real linear space $Y$ is a nonempty set $K\subseteq Y$ satisfying $\lambda K\subseteq K$ for every $\lambda\geq0$; in particular, $0_Y\in K$. A cone is pointed if $K\cap(-K)=\{0_Y\}$ and solid if $\interior{K}\neq\emptyset$.

To prove some equivalent descriptions of the set of properly efficient points in Section \ref{s-propeff}, we will use the following assertions.

    \begin{Lemma}[{\cite[1.1.6~Theorem]{Engelking1989Generaltopologybook}}]\label{th:Topologie:IntIntersection=IntersectionInt}
		Let $\T$ be a topological space, $k \in \N$ and for all $i \in \{1, 2, \dotsc, k\}$ let $A_i \subseteq \T$.
		Then
		\[\interior{\bigcap_{i \mathop = 1}^k A_i}= \bigcap_{i \mathop = 1}^k \interior{A_i}.\]
	\end{Lemma}

	\begin{Lemma}[{\cite[Prop.~7.1.2]{Lopez2024PointSetTopologyWorkingTextbookbook}}] \label{th:Topologie:InneresVonKreuzprodukt}
		Assume that $\T_1,\T_2$ are real linear-topological spaces and $A_1\subseteq \T_1$, $A_2\subseteq \T_2$ are convex sets. Then
		\[\interior[\T_1\times\T_2]{A_1\times A_2}=\interior[\T_1]{A_1}\times\interior[\T_2]{A_2}.\]
	\end{Lemma}

    The following assertions play an important role in the proof of the relationships between the sets of properly efficient solutions to the non-penalised and penalised vector problems in Section \ref{sec:vectorialPenalisation:PEff:AddRemove}.

	\begin{Proposition}\label{th:Topologie:IntCone0IsCone}\label{th:Topologie:IntIntCone0=IntCone}
		Let $\T$ be a linear-topological space and $C\subseteq\T$ a cone. Then $\interior{\interior{C}\cup\{0_T\}}=\interior{C}$. 
        If $C$ is convex, then also $\interior{C}\cup\{0_T\}$ is a convex cone. As a consequence, if $C$ is convex, then $0_\T\in\interior{C}$ if and only if $C=\T$. In particular, $\interior{C}\cup\{0_T\}$ is pointed whenever $C$ is convex and $C\neq\T$.
	\end{Proposition}

	We recall two elementary properties of convex cones that will be used in the construction of dilating cones.

	\begin{Proposition}[{\cite[Lemmas~1.11, 1.12, 1.32]{Jahn2011VectorOptimizationTheoryApplicationsExtensionsbook}, \cite{Zalinescu2002ConvexAnalysisGeneralVectorSpacesbook}}]\label{th:Mengeneigenschaften:Kegel:konvexÄquivalenz}
		Let $\T$ be a linear-topological space and $C\subseteq\T$ a cone. Then $C$ is convex if and only if $C+C\subseteq C$. If $C$ is convex and $\interior[\T]{C}\neq\emptyset$, then
        \[
            C+\interior[\T]{C}=\interior[\T]{C}.
        \]
        If $K\subseteq\T$ is another convex cone, then also $C\cap K$ is a convex cone.
	\end{Proposition}

	\begin{Proposition}\label{th:Mengeneigenschaften:Kegel:InteriorNegativeConeDisjoint}
		Let $\T$ be a linear-topological space and $C\subseteq\T$ a convex cone with $C\neq\T$. Then
		\[
			C\cap\interior[\T]{-C}=\emptyset.
		\]
	\end{Proposition}
	\begin{proof}
	Suppose $x\in C\cap\interior[\T]{-C}$. Since negation is a homeomorphism, $-x\in\interior[\T]{C}$. By \Cref{th:Mengeneigenschaften:Kegel:konvexÄquivalenz}, $C+\interior[\T]{C}=\interior[\T]{C}$, so $0_\T=x+(-x)\in\interior[\T]{C}$. Then \Cref{th:Topologie:IntCone0IsCone} implies $C=\T$, contradicting the assumption.
	\end{proof}

    The following change of variables will allow us to describe the cones associated with the penalised problem as images of product sets.

    \begin{Proposition}\label{th:Topology:ProductSpace:x+thContinuous}
		Let $\T$ be a linear-topological space and $h\in\T$. Define $f:\T\times\R\to\T\times\R$ by $f(x,\lambda)=(x+\lambda h,\lambda)$. Then $f$ is a linear homeomorphism.
	\end{Proposition}
	\begin{proof}
		The map $f$ is linear and continuous by continuity of addition and scalar multiplication. Its inverse is
		\[
			f^{-1}(x,\lambda)=(x-\lambda h,\lambda),
		\]
		which is continuous for the same reason.
	\end{proof}

    To compare cones in the original and extended objective spaces, we introduce notation for projections and for sections obtained by fixing the second coordinate.

		\begin{Definition}
		Let $X,Y$ be sets, let $S\subseteq X\times Y$. The \defi{projection of $S$ onto $X$} is defined by
		\[\ProjectionSet{X}{S}\coloneqq\Menge{x\in X}{\exists y\in Y:\ (x,y)\in S}.\]
	\end{Definition}

	\begin{Definition}
        Let $X,Y$ be sets, $S\subseteq X\times Y$ and $y\in Y$. The \defi{section of $S$ at $y$} (with respect to the second coordinate) is defined to be the set
		\begin{align*}
			\SetRestrictionToPoint{S}{(X,y)}\coloneqq\ProjectionSet{X}{S\cap\left(X\times\{y\}\right)}=\Menge{x\in X}{(x,y)\in S}.
		\end{align*}
	\end{Definition}

    The next results prepare the comparison between taking the interior of a set and taking a section at a fixed second coordinate.

	\begin{Proposition}\label{th:Topologie:ProduktTopologie:RestrictionOfOpenSetIsOpen}
		Let $\T_1,\T_2$ be topological spaces, let $S\subseteq\T_1\times\T_2$ be open, and let $y\in\T_2$. Then $\SetRestrictionToPoint{S}{(\T_1,y)}$ is open in $\T_1$.
	\end{Proposition}
	\begin{proof}
		For any $(x,y)\in S$ there exist open neighbourhoods $U\subseteq \T_1$ of $x$ and $V\subseteq \T_2$ of $y$ such that $U\times V\subseteq S$. Thus, it is $U\subseteq \SetRestrictionToPoint{S}{(\T_1,y)}$ and therefore $\SetRestrictionToPoint{S}{(\T_1,y)}$ open.
	\end{proof}
    
	To relate the interior of a convex set to the interiors of its sections, we also recall the following property of line segments.
    
	\begin{Lemma}[{\cite[Chap.~2~1.1]{SchaeferWolff1999TopologicalVectorSpacesbook}}]\label{th:Topologie:ConvexSetHalfopenLinesInInterior}
		Let $\T$ be a linear-topological space and $K\subseteq\T$ be convex. For any $x\in\interior{K}$ and any $y\in\cl{K}$ there exists $z\in\interior{K}$ such that $x\in(z,y)\subseteq\interior{K}$. In particular, it is
		\[[x,y)\subseteq\interior{K}.\]
	\end{Lemma}

 The following proposition will be applied in the proof of the properties of dilating cones in Section \ref{sec:vectorialPenalisation:PEff:AddRemove}.

 \begin{Proposition}\label{th:Topologie:ProductTopology:RestrictionOfInteriorOfSet=InteriorOfRestriction}
		Let $\T_1,\T_2$ be topological spaces. For any set $S\subseteq \T_1\times \T_2$ and any $y\in \T_2$ it is
		\[\SetRestrictionToPoint{\interior[\T_1\times\T_2]{S}}{(\T_1,y)}\subseteq\interior[\T_1]{\SetRestrictionToPoint{S}{(\T_1,y)}}.\]
		If $\T_1$ and $\T_2$ are linear-topological spaces, $S$ is convex and there exists $x\in \T_1$ with $(x,y)\in\interior[\T_1\times \T_2]{S}$ then
		\[\SetRestrictionToPoint{\interior[\T_1\times \T_2]{S}}{(\T_1,y)}=\interior[\T_1]{\SetRestrictionToPoint{S}{(\T_1,y)}}.\]
	\end{Proposition}
	\begin{proof}
		By \Cref{th:Topologie:ProduktTopologie:RestrictionOfOpenSetIsOpen}, the set $\SetRestrictionToPoint{\interior[\T_1\times\T_2]{S}}{(\T_1,y)}$ is open in $\T_1$. Since it is contained in $\SetRestrictionToPoint{S}{(\T_1,y)}$, the first inclusion follows.

        Now suppose that $\T_1$ and $\T_2$ are linear-topological spaces, that $S$ is convex, and that $(x,y)\in\interior[\T_1\times\T_2]{S}$. Let $w\in\interior[\T_1]{\SetRestrictionToPoint{S}{(\T_1,y)}}$. By the first inclusion, $x\in\interior[\T_1]{\SetRestrictionToPoint{S}{(\T_1,y)}}$. Applying \Cref{th:Topologie:ConvexSetHalfopenLinesInInterior} within the section, we obtain $z\in\interior[\T_1]{\SetRestrictionToPoint{S}{(\T_1,y)}}$ such that $w\in(z,x)\subseteq \interior[\T_1]{\SetRestrictionToPoint{S}{(\T_1,y)}}$. Hence $(z,y)\in S$, and a second application of \Cref{th:Topologie:ConvexSetHalfopenLinesInInterior} in $\T_1\times\T_2$ yields
        \[
            (w,y)\in[(x,y),(z,y))\subseteq\interior[\T_1\times\T_2]{S}.
        \]
        Thus $w\in\SetRestrictionToPoint{\interior[\T_1\times\T_2]{S}}{(\T_1,y)}$, proving the reverse inclusion.
		%
		%
		%
		%
	\end{proof}

    The following counterexample shows that the equality in the preceding proposition can fail when $S$ is not convex.

	\begin{Beispiel}
		The last result is in general not true for non-convex sets even if they have nonempty interior. Consider the Saturn set (based on the planet)
		\[S\coloneqq \B_{\R^3}(0_{\R^3},1)\cup\left(\B_{\R^2}(0_{\R^2},2)\times\{0\}\right)\subseteq\R^2\times\R.\]
		Then it is
		\begin{align*}
			&\interior[\R^2]{\SetRestrictionToPoint{S}{(\R^2,0)}}=\interior[\R^2]{\B_{\R^2}(0_{\R^2},2)}\\
			&\neq\interior[\R^2]{\B_{\R^2}(0_{\R^2},1)}=\SetRestrictionToPoint{\interior[\R^3]{\B_{\R^3}(0_{\R^3},1)}}{(\R^2,0)}=\SetRestrictionToPoint{\interior[\R^3]{S}}{(\R^2,0)}.
		\end{align*}
	\end{Beispiel}

    In order to compare elements of a real linear space, we introduce the following binary relations.

    \begin{Definition}\label{def:<ARelationen}
        For a real linear space $Y$, $x,y\in Y$ and a set $\emptyset\neq A\subseteq Y$ we define the binary relations
    	\begin{align*}
        	x\leq_A y\quad&:\Leftrightarrow\quad y\in x+A, \label{gl:Kegelleq}\\
        	x\lneq_A y\quad&:\Leftrightarrow\quad y\in x+A\ohne\{0_Y\},\notag\\
        	x<_A y\quad&:\Leftrightarrow\quad y\in x+\interior{A},\notag
    	\end{align*}
        where for $\lneq_A$ we assume $A\neq\{0_Y\}$ and for $<_A$ we assume $\interior{A}\neq\emptyset$.
    \end{Definition}

We will study a penalisation approach in Section \ref{sec:vectorialPenalisation:PEff} where we suppose that certain properties concerning the level set of the penalisation function are satisfied.

    \begin{Definition}[Level sets]
        Let $X,Y$ be two sets and let $S\subseteq X$ and $\prec_A\subseteq Y\times Y$ be a binary relation (as in \Cref{def:<ARelationen}). Then we define the \defi{$\prec_A$-level set} of a function $f:X\to Y$ on a set $S\subseteq X$ to $y\in Y$ as
        \[\lev_{\prec_A}(S,f,y)\coloneqq\Menge{s\in S}{f(s)\prec_A y}.\]
    \end{Definition}

	\subsection{Generalised cone convexity} \label{sec:gen-cone-convexity}

    Let us recall two important concepts of generalised convexity for vector functions, which will play a key role in our paper.

	\begin{Definition} \label{def:Aconvexity-and-weakAquasiconvexity}
		Let $X,Y$ be linear spaces, $\calD\subseteq X$, $\s\subseteq \calD$ be a convex set, and $A\subseteq Y$. A function $f:\calD\to Y$ is called \defi{$A$-convex} on $\s$ if
		\[\forall x^1,x^2\in\s,\ \forall\lambda\in(0,1):\ f\left(\lambda x^1+(1-\lambda)x^2\right)\leq_A \lambda f\left(x^1\right)+(1-\lambda)f\left(x^2\right).\]
		$f$ is called \defi{weakly $A$-quasiconvex} on $\s$ (in the sense of \textcite{Jahn2011VectorOptimizationTheoryApplicationsExtensionsbook}) if
		\[\forall x^1,x^2\in\s,\ x^1\neq x^2,\ f\left(x^1\right)\leq_A f\left(x^2\right);\ \exists x^0\in\s\setminus\{x^2\};\ \forall x\in[x^0,x^2):\ f(x)\leq_A f\left(x^2\right).\]
	\end{Definition}
    
    The next observation allows us to pass from convexity with respect to the ordering cone to convexity with respect to a larger set.
    
    \begin{Lemma}\label{th:LineareRäume:Funktionseigenschaften:Konvexität:Vektorwertig:AKonvexA1A2}
		Let $X$ and $Y$ be linear spaces, let $S\subseteq X$ be convex, let $f:X\to Y$, and let $A_1\subseteq A_2\subseteq Y$. If $f$ is $A_1$-convex on $S$, then it is also $A_2$-convex on $S$.
	\end{Lemma}
	\begin{proof}
		The assertion follows directly from the definition.
	\end{proof}
    
    In particular, we are interested in the concepts from \Cref{def:Aconvexity-and-weakAquasiconvexity} with $A \in \{ \interior{C}, C \setminus \{0_Y\}, C\}$ (where $C\neq\{0_Y\}$ is a pointed, convex cone).
	The following simple relationship between these two (generalised) cone convexity concepts can be found in \textcite{Jahn2011VectorOptimizationTheoryApplicationsExtensionsbook}.

	\begin{Lemma}\label{th:CKonvImplik}
		 Let $X,Y$ be real linear spaces, let $\calD\subseteq X$, let $f:\calD\to Y$, let $\s\subseteq\calD$ be convex, and let $A\subseteq Y$ satisfy $A+\lambda A\subseteq A$ for every $\lambda\in(0,1)$. If $f$ is $A$-convex on $\s$, then it is also weakly $A$-quasiconvex on $\s$. Moreover, if $C\subseteq Y$ is a pointed convex cone and $f$ is $C$-convex on $\s$, then $f$ is weakly $C\setminus\{0_Y\}$-quasiconvex on $\s$.
	\end{Lemma}
	\begin{proof}
        Let $x^1,x^2\in\s$, $x^1\neq x^2$, satisfy $f(x^1)\leq_A f(x^2)$, and set $x^0\coloneqq x^1$. The required inequality at $x=x^0$ holds by assumption. Now let $x=\lambda x^0+(1-\lambda)x^2\in(x^0,x^2)$, where $\lambda\in(0,1)$. By $A$-convexity and the assumption on $A$,
        \[
            f(x)\in \lambda f(x^0)+(1-\lambda)f(x^2)-A\subseteq\lambda\left(f(x^2)-A\right)+(1-\lambda)f(x^2)-A\subseteq f(x^2)-A.
        \]
		Hence $f(x)\leq_A f(x^2)$ for every $x\in[x^0,x^2)$, proving the first assertion.

		For the second assertion, let $x^1,x^2\in\s$, $x^1\neq x^2$, satisfy $f(x^1)\lneq_C f(x^2)$, and set $x^0\coloneqq x^1$. The required relation holds at $x=x^0$. Let $x=\lambda x^0+(1-\lambda)x^2\in(x^0,x^2)$ with $\lambda\in(0,1)$. By $C$-convexity,
        \[\lambda f(x^0)+(1-\lambda)f(x^2)-f(x)\in C.\]
        Since $C$ is convex, \Cref{th:Mengeneigenschaften:Kegel:konvexÄquivalenz} and the pointedness of $C$ imply
        \[
			f(x^2)-f(x)
			=\lambda\bigl(f(x^2)-f(x^1)\bigr)
			 +\lambda f(x^1)+(1-\lambda)f(x^2)-f(x)
			\in C\setminus\{0_Y\}+C\subseteq C\setminus\{0_Y\}.
		\]
		Thus, $f$ is weakly $C\setminus\{0_Y\}$-quasiconvex on $\s$.
	\end{proof}

    The condition on $A$ in \Cref{th:CKonvImplik} is satisfied whenever $A$ is convex and $A+A\subseteq A$, the latter condition being equivalent to the transitivity of $\leq_A$ on $Y$. In particular, if $C\subseteq Y$ is a convex cone, then the condition holds for $A\in\{C,\interior{C}\cup\{0_Y\},\interior{C}\}$. It also holds for $A=C\setminus\{0_Y\}$ when $C$ is pointed. The converse of the implication in \Cref{th:CKonvImplik} is false in general.

	\section{Solution concepts in vector optimisation}\label{s-proper eff}

    Throughout the paper, the following main assumptions are made:
	\begin{align}
		\begin{cases}
			\meinUrbildraum,\meinBildraum\text{ real linear-topological spaces},\\
			\emptyset\neq\calD\subseteq\meinUrbildraum,\ f:\calD\to\meinBildraum,\\
			\{0_Y\}\neq\c\subseteq\meinBildraum\ \text{pointed, convex cone},\\
			\emptyset\neq\s\subseteq\calD.
		\end{cases}\tag{$A$}\label{ass:VOPRestrVor}
	\end{align}
    Under assumption \eqref{ass:VOPRestrVor}, we consider the vector optimisation problem
	\begin{equation}
		\argminKegel{\c}\limits_{x\in\s}\ f(x). \tag{$P_\s$} \label{P:P}
	\end{equation}

	In the following definition, we introduce the classical concepts of solutions to the vector optimisation problem \eqref{P:P} with respect to a general subset $\cal C$ of $Y$. To formulate these concepts, we replace the nontrivial, pointed, convex cone $\c$ in the assumption \eqref{ass:VOPRestrVor} by a general subset $\cal C$ of $Y$ and denote this assumption by ($A^{\prime}$). Furthermore, we consider $\cal C$ instead of $\c$ in \eqref{P:P} under the assumption ($A^{\prime}$).

	\begin{Definition}[Efficiency]
		\label{def:Efficiency}
		Consider the problem \eqref{P:P} under the assumption ($A^{\prime}$). A point $\x \in \s$ is said to be an \defi{efficient solution} of \eqref{P:P} if $(f[\s] - f(\x)) \cap (-{\cal C}) \subseteq \{0_\meinBildraum\}$. The set of all efficient solutions of \eqref{P:P} is denoted by
		\[\Eff{f}{\s}{{\cal C}}\coloneqq\Menge{\x\in\s}{\left(f[\s] - f(\x)\right) \cap (-{\cal C}) \subseteq \{0_\meinBildraum\}}.\]
	\end{Definition}

	\begin{Definition}[Weak efficiency]
		\label{def:WeakEfficiency}
		Consider the problem \eqref{P:P} under the assumption ($A^{\prime}$). A point $\x \in \s$ is said to be a \defi{weakly efficient solution} of \eqref{P:P} if $(f[\s] - f(\x)) \cap (-\interior[\meinBildraum]{{\cal C}})\subseteq\{0_\meinBildraum\}$. The set of all weakly efficient solutions of \eqref{P:P} is denoted by
		\[\WEff{f}{\s}{{\cal C}}\coloneqq\Menge{\x\in\s}{\left(f[\s] - f(\x)\right) \cap (-\interior[\meinBildraum]{{\cal C}})\subseteq\{0_\meinBildraum\}}.\]
	\end{Definition}

	In addition to these solution concepts, we will examine properly efficient elements and relationships between the different solution concepts in Section \ref{s-propeff}.
	\subsection{Proper efficiency}\label{s-propeff}

	Proper efficiency is an extension of the well-known efficiency concepts, that is particularly relevant when computing numerical solutions (see \cite{Geoffrion1968ProperefficiencytheoryvectormaximizationjournalArticle,Hartley1978ConeEfficiencyConeConvexityConeCompactnessjournalArticle}) or when using scalarisation methods to find solutions (see \cite{GuntherKhazayelTammer2022VectorOptimizationRelativelySolidConvexConesRealLinearSpacesjournalArticle}, \cite[Chapter~5]{KhanTammerZalinescu2015SetvaluedOptimizationIntroductionApplicationsbook} and references therein). It is well known that many solution methods based on scalarisation generate only weakly efficient solutions, which can be far away from the next efficient solution. An overview of different concepts of proper efficiency is given in \cite[Section~2.4]{KhanTammerZalinescu2015SetvaluedOptimizationIntroductionApplicationsbook}.
 However, for several proper efficiency concepts there exist density results that state that, under certain conditions, the set of properly efficient points is dense in the set of efficient points (see \cite{ArrowBarankinBlackwell19535AdmissiblePointsConvexSetsbookSection}, \cite{NewhallGoodrich2015DensityHenigEfficientPointsLocallyConvexTopologicalVectorSpacesjournalArticle}, \cite{GopfertTammerZalinescu2004newABBtheoremnormedvectorspacesjournalArticle}).

	As shown in \cite{GuerraggioMolhoZaffaroni1994notionproperefficiencyvectoroptimizationjournalArticle,Henig1982ProperefficiencyrespectconesjournalArticlea} and \cite[Section~2.4]{KhanTammerZalinescu2015SetvaluedOptimizationIntroductionApplicationsbook}, under certain conditions many proper efficiency notions coincide with the concept introduced by \textcite{Henig1982ProperefficiencyrespectconesjournalArticlea}. He defined a point as being properly efficient if it is efficient with respect to a larger convex cone, the so-called dilating cone.

	\begin{Definition}
		Assume \eqref{ass:VOPRestrVor}. A convex cone $H\subset\meinBildraum$ is said to be a \defi{Henig dilating cone} of $\c$ if $H\neq\meinBildraum$ and $\c\setminus\{0_\meinBildraum\}\subseteq\interior{H}$. The set of Henig dilating cones of $\c$ is defined as
		\[\HenigDilatingCones{\c}\coloneqq\Menge{H\subset\meinBildraum}{H\neq\meinBildraum,\ H\ \text{convex cone},\ \c\ohne\{0_\meinBildraum\}\subseteq\interior{H}}.\]
	\end{Definition}

    Dilating cones have been used to derive necessary optimality conditions for Pareto minima in Banach spaces, where the ordering cone may have empty interior \cite{DureaDutta2008}, and to regularise multiobjective optimisation and optimal-control problems by replacing the ordering cone with solid cones \cite{JadambaKhanLopezSama2019ConicalregularizationmultiobjectiveoptimizationproblemsjournalArticle,KhanSama2013NewConicalRegularizationOptimizationOptimalControlProblemsConvergenceAnalysisFiniteElementDiscretizationjournalArticle}.
	
	Some basic properties of the class of dilating cones are collected in the following results.

    \begin{Proposition}\label{th:PEff:HenigDilatingConesProperties}
        Assume \eqref{ass:VOPRestrVor}. The family of Henig dilating cones has the following properties:
        \begin{enumerate}
            \item\label{item:PEff:HenigDilatingConesProperties:FiniteIntersection} For every $n\in\N\setminus\{0\}$ and $H_1,\dotsc,H_n\in\HenigDilatingCones{\c}$,
            \[\bigcap_{i=1}^n H_i\in\HenigDilatingCones{\c}.\]

            \item\label{item:PEff:HenigDilatingConesProperties:OrderingConeMonotonicity} If $\c_1\subseteq\c_2$ are nontrivial pointed convex cones, then
            $\HenigDilatingCones{\c_2}\subseteq\HenigDilatingCones{\c_1}$.

            \item\label{item:PEff:HenigDilatingConesProperties:UpwardClosure} If $H\in\HenigDilatingCones{\c}$ and $K$ is a convex cone with
            $H\subseteq K\neq\meinBildraum$, then $K\in\HenigDilatingCones{\c}$.

            \item\label{item:PEff:HenigDilatingConesProperties:OpenPointed} If $H\in\HenigDilatingCones{\c}$, then
            $\interior{H}\cup\{0_\meinBildraum\}\in\HenigDilatingCones{\c}$ and this cone is pointed.
        \end{enumerate}
    \end{Proposition}
    \begin{proof}
		\begin{enumerate}[leftmargin=*]
			\item[\ref{item:PEff:HenigDilatingConesProperties:FiniteIntersection}]
			The intersection is a convex cone and is not the whole space, since it is contained in $H_1\neq\meinBildraum$. Moreover, by \Cref{th:Topologie:IntIntersection=IntersectionInt},
			\[
				\c\setminus\{0_\meinBildraum\}
				\subseteq\bigcap_{i=1}^n\interior{H_i}
				=\interior{\bigcap_{i=1}^n H_i}.
			\]

			\item[\ref{item:PEff:HenigDilatingConesProperties:OrderingConeMonotonicity}]
			The defining interior inclusion for $\c_2$ implies the one for $\c_1$.

			\item[\ref{item:PEff:HenigDilatingConesProperties:UpwardClosure}]
			$\interior{H}\subseteq\interior{K}$, so the defining inclusion is preserved.

			\item[\ref{item:PEff:HenigDilatingConesProperties:OpenPointed}]
			The cone property, convexity, and the identity $\interior{\interior{H}\cup\{0_\meinBildraum\}}=\interior{H}$ follow from \Cref{th:Topologie:IntCone0IsCone}. That result also implies $\interior{H}\cup\{0_\meinBildraum\}$ is pointed because $H\neq\meinBildraum$, and hence belongs to $\HenigDilatingCones{\c}$. \qedhere
		\end{enumerate}
    \end{proof}

	\begin{Definition}[Henig-proper efficiency \cite{Henig1982ProperefficiencyrespectconesjournalArticlea}] \label{def:HenigProperEfficiency}
		Consider \eqref{P:P} under \eqref{ass:VOPRestrVor}. A point $\x \in \s$ is said to be a \defi{Henig-properly efficient solution} if $\x \in \Eff{f}{\s}{H}$ for some $H \in \HenigDilatingCones{\c}$.  The set of \defi{Henig-properly efficient} points of $f$ and $\s$ with respect to $\c$ is defined as
		\begin{align*}
			\PEffHenig{f}{\s}{\c}\coloneqq{}&\Menge{\x\in\s}{\exists H\in\HenigDilatingCones{\c}:\ \x\in\Eff{f}{\s}{H}} \\
			={}&\bigcup\limits_{H\in\HenigDilatingCones{\c}} \Eff{f}{\s}{H}.
		\end{align*}
	\end{Definition}

    \begin{Bemerkung}
        The notion of Henig-proper efficiency given in \Cref{def:HenigProperEfficiency} is a generalisation of other known proper efficiency concepts, for example in normed spaces or with the standard ordering cone.
        If $\c$ is not pointed, then $\HenigDilatingCones{\c}=\emptyset$, because for every convex cone $H$ with $C\setminus\{0_Y\}\subseteq \interior{H}$ it would be $0_\meinBildraum\in\interior{H}$ and therefore by \Cref{th:Topologie:IntCone0IsCone} $H=\meinBildraum$, contrary to the definition. As a direct consequence it is $\PEffHenig{f}{\s}{\c}=\emptyset$. Such cases may instead be treated by considering $\HenigDilatingCones{(\c\setminus\lin(\c))\cup\{0_\meinBildraum\}}$ (as in \cite{Luc1989TheoryVectorOptimizationbook}), where $\lin(\c)\coloneqq \c\cap(-\c)$, or by not requiring $H$ to be convex and instead taking $H=Y\setminus \tilde H\cup\{0_Y\}$ with a convex cone $\tilde H$ and still requiring $C\setminus\{0_Y\}\subseteq\interior{H}$.
    \end{Bemerkung}

    The following example explains why the whole image space is excluded from the family of dilating cones.
	\begin{Beispiel}
		Consider the cone $C=\R_\geq\times\R\subseteq\R^2$ (or any other halfspace or lexicographic cone). If a convex cone $H$ satisfies $C\setminus\{0_{\R^2}\}\subseteq\interior{H}$, then $(0,-1),(0,1)\in\interior{H}$. Hence $0_{\R^2}\in\interior{H}$ and thus $H=\R^2$ by \Cref{th:Topologie:IntCone0IsCone}. Consequently, $\HenigDilatingCones{C}=\emptyset$ under the preceding definition.

	\end{Beispiel}

    The exclusion of the whole space is essential even when the ordering cone is pointed, but does not have a bounded convex base and is not open. 
    
    \begin{Beispiel}
        Let
        \[
			C_{\mathrm{lex}}\coloneqq\Menge{(u,v)\in\R^2}{u>0\ \text{or}\ (u=0\ \text{and}\ v\geq0)}
		\]
		be the lexicographic cone. Every convex cone $H$ satisfying $C_{\mathrm{lex}}\setminus\{0_{\R^2}\}\subseteq\interior{H}$ equals $\R^2$. Indeed, since $(0,1)\in\interior{H}$ and $\interior{H}$ is open, there exists $(-a,b)\in\interior{H}$ with $a,b>0$. As $(a,-b)\in C_{\mathrm{lex}}\setminus\{0_{\R^2}\}\subseteq\interior{H}$, we obtain $0_{\R^2}\in\interior{H}$ and hence $H=\R^2$. Thus $\HenigDilatingCones{C_{\mathrm{lex}}}=\emptyset$.
    \end{Beispiel}

	The following relationships between different solution concepts follow immediately from their definitions.

	\begin{Lemma}\label{th:BeziehungEff}\label{th:EffizienzÄquivalenzen}
		Suppose \eqref{ass:VOPRestrVor}. Then, the following assertions hold:
		\begin{enumerate}
			\item $\PEffHenig{f}{\s}{\c}\subseteq\Eff{f}{\s}{\c}\subseteq\WEff{f}{\s}{\c}$.
			\item \label{item:EffizienzÄquivalenzen:WeffC=EffIcoreC}
			$\WEff{f}{\s}{\c}=\Eff{f}{\s}{\interior[\meinBildraum]{\c}\cup\{0_\meinBildraum\}}$.
			\item If $\c=\interior[\meinBildraum]{\c}\cup\{0_\meinBildraum\}$ it is $\Eff{f}{\s}{\c}=\WEff{f}{\s}{\c}=\PEffHenig{f}{\s}{\c}$.
			\item \label{item:BeziehungEff:Eff(C2)inEff(C1)}
			If $\emptyset\neq\c_1\subseteq\c_2\subseteq\meinBildraum$ are arbitrary sets, then $\Eff{f}{\s}{\c_2}\subseteq\Eff{f}{\s}{\c_1}$ and $\WEff{f}{\s}{\c_2}\subseteq\WEff{f}{\s}{\c_1}$.
		\end{enumerate}
	\end{Lemma}

    Now we prove some equivalent descriptions of the set of Henig-properly efficient points involving weakly efficient points and special subsets of $\HenigDilatingCones{\c}$.

	\begin{Proposition}\label{th:PEff:PEffHenigEqui}
		Assume \eqref{ass:VOPRestrVor}. Then it is
		\begin{equation}
			\PEffHenig{f}{\s}{\c}=\bigcup\limits_{H\in\HenigDilatingCones{\c}} \Eff{f}{\s}{H}=\bigcup\limits_{H\in\HenigDilatingCones{\c}} \WEff{f}{\s}{H}. \label{gl:PEff:PEffHenigEqui:Eff=WEff}
		\end{equation}
		Furthermore, we have
		\begin{equation}
			\PEffHenig{f}{\s}{\c}=\bigcup\limits_{\substack{H\in\HenigDilatingCones{\c}\\ H=\interior{H}\cup\{0_\meinBildraum\}}} \Eff{f}{\s}{H}=\bigcup\limits_{\substack{H\in\HenigDilatingCones{\c}\\ H=\interior{H}\cup\{0_\meinBildraum\}}} \WEff{f}{\s}{H} \label{gl:PEff:PEffHenigEqui:OnlyOpenH}
		\end{equation}
		and for every $\calH\in\HenigDilatingCones{\c}$ it is also
		\begin{align}
			\PEffHenig{f}{\s}{\c}&=\bigcup\limits_{\substack{H\in\HenigDilatingCones{\c}\\ H\subseteq\calH}} \Eff{f}{\s}{H}=\bigcup\limits_{\substack{H\in\HenigDilatingCones{\c}\\ H\subseteq\calH}} \WEff{f}{\s}{H}\\
			&=\bigcup\limits_{H\in\HenigDilatingCones{\c}} \Eff{f}{\s}{H\cap\calH}=\bigcup\limits_{H\in\HenigDilatingCones{\c}} \WEff{f}{\s}{H\cap\calH}.
		\end{align}
	\end{Proposition}
	\begin{proof}
		The first equality is just the definition. Since $\Eff{f}{\s}{H}\subseteq\WEff{f}{\s}{H}$ we also get the forward inclusion. For the other one, take any $H\in\HenigDilatingCones{\c}$ and set $J\coloneqq\interior{H}\cup\{0_\meinBildraum\}$. By \Cref{th:PEff:HenigDilatingConesProperties}~\ref{item:PEff:HenigDilatingConesProperties:OpenPointed}, $J\in\HenigDilatingCones{\c}$. By \Cref{th:EffizienzÄquivalenzen}~\ref{item:EffizienzÄquivalenzen:WeffC=EffIcoreC}, we have $\WEff{f}{\s}{H}=\Eff{f}{\s}{J}$, which proves the required inclusion and thus \eqref{gl:PEff:PEffHenigEqui:OnlyOpenH}.\\
		With the same argument we can see
		\[\PEffHenig{f}{\s}{\c}\supseteq\bigcup\limits_{\substack{H\in\HenigDilatingCones{\c}\\ H\subseteq\calH}} \Eff{f}{\s}{H}=\bigcup\limits_{\substack{H\in\HenigDilatingCones{\c}\\ H\subseteq\calH}} \WEff{f}{\s}{H}.\]
		For the inverse inclusion, let $H\in\HenigDilatingCones{\c}$. By \Cref{th:BeziehungEff}~\ref{item:BeziehungEff:Eff(C2)inEff(C1)} we get $\Eff{f}{\s}{H}\subseteq\Eff{f}{\s}{H\cap\calH}$. Moreover, \Cref{th:PEff:HenigDilatingConesProperties}~\ref{item:PEff:HenigDilatingConesProperties:FiniteIntersection} gives $H\cap\calH\in\HenigDilatingCones{\c}$, proving the remaining inclusions.
	\end{proof}


	\Cref{th:PEff:PEffHenigEqui} allows us to only consider special forms of dilating cones. We will, for example, consider only dilating cones of the form $H=\interior[\meinBildraum]{H}\cup\{0_\meinBildraum\}$ in some of the later proofs.

	\subsection{Relationships between constrained and unconstrained vector optimisation problems}
    
	As the general goal of our paper is to express the set of properly efficient solutions of \eqref{P:P} by properly efficient solutions of unconstrained problems, we introduce the corresponding unconstrained vector optimisation problem here under the same assumption \eqref{ass:VOPRestrVor}
	\begin{equation}
		\argminKegel{\c}\limits_{x\in\calD}\ f(x). \tag{$P_\calD$} \label{P:PX}
	\end{equation}
	We will analyse important relationships between constrained and unconstrained vector optimisation problems by using cone convexity concepts for vector functions.

	As a first result, we show that the intersection of $\s$ and the set of Henig-properly efficient solutions to the unconstrained problem is a subset of the set of Henig-properly efficient solutions to the constrained problem.

	\begin{Satz} \label{th:MOP:Eff:PEff:ScapPEffInPEff}
		Assume \eqref{ass:VOPRestrVor}. Then, for every set $\calU\subseteq X$ with $\s\subseteq \calU\subseteq\calD$, we have
		\begin{align}
			\s\cap\PEffHenig{f}{\calD}{\c}&\subseteq\PEffHenig{f}{\calU}{\c}. \label{gl:MOP:Eff:PEff:ScapPEffInPEff}
		\end{align}
	\end{Satz}
	\begin{proof}
		Follows directly from the definition of Henig-proper efficiency and \cite[Theorem~3.1]{GuntherKobisSchmollingTammer2023VectorialpenalisationvectoroptimisationreallineartopologicalspacesjournalArticle} giving us
        \[\s\cap\PEffHenig{f}{\calD}{\c}=\s\cap\bigcup_{H\in\HenigDilatingCones{\c}} \Eff{f}{\calD}{H}\subseteq\bigcup_{H\in\HenigDilatingCones{\c}} \Eff{f}{\calU}{H}=\PEffHenig{f}{\calU}{\c}.\qedhere\]
	\end{proof}

	Now, we use the introduced cone convexity concept (from \Cref{sec:gen-cone-convexity}) to show the reverse inclusion of \eqref{gl:MOP:Eff:PEff:ScapPEffInPEff} on open sets, allowing us to split the treatment of the constrained set into its interior and its boundary, which gives rise to a barrier for the solution set of \eqref{P:P}.

	\begin{Satz}\label{th:MOP:Eff:PEff:LocalSolutionsInIntAreGlobalIfConvex}
		Assume \eqref{ass:VOPRestrVor}, let $\calD$ be convex and let $f$ be $\c$-convex on $\calD$. Then
		\begin{gather}
			\interior{\s}\cap\PEffHenig{f}{\s}{\c}\subseteq\interior{\s}\cap\PEffHenig{f}{\calD}{\c}, \label{gl:MOP:Eff:PEff:IntCapPEff(S)InIntCapPEff(D)}\\
			\s\cap\PEffHenig{f}{\calD}{\c}\subseteq\PEffHenig{f}{\s}{\c}\subseteq\big(\interior{\s}\cap\PEffHenig{f}{\calD}{\c}\big)\cup\big(\s\cap\boundary[\meinUrbildraum]{\s}\big), \label{gl:SchrankePEffInt}\\
			\PEffHenig{f}{\s}{\c}\setminus\PEffHenig{f}{\calD}{\c}\subseteq \s\cap\boundary[\meinUrbildraum]{\s}. \label{gl:MOP:Eff:PEff:PEff/PEffInBD}
		\end{gather}
		If additionally $\s$ is open or $\PEffHenig{f}{\s}{\c}\subseteq\interior[\meinUrbildraum]{\s}$, then
        \begin{equation}
            \s\cap\PEffHenig{f}{\calD}{\c}=\PEffHenig{f}{\s}{\c}. \label{gl:MOP:Eff:PEff:SPEffD=PEffS}
        \end{equation}
	\end{Satz}
	\begin{proof}
		For $\x\in\interior{\s}\cap\PEffHenig{f}{\s}{\c}$ by \Cref{th:PEff:PEffHenigEqui}~\eqref{gl:PEff:PEffHenigEqui:Eff=WEff} there exists $H\in\HenigDilatingCones{\c}$ with $\x\in\WEff{f}{\s}{H}$. Set $J\coloneqq\interior{H}\cup\{0_\meinBildraum\}$. Since $H\in\HenigDilatingCones{\c}$, we have $\c\subseteq J$, and hence, by \Cref{th:LineareRäume:Funktionseigenschaften:Konvexität:Vektorwertig:AKonvexA1A2}, $f$ is $J$-convex on $\calD$. Moreover, $J\in\HenigDilatingCones{\c}$ is a pointed convex cone by \Cref{th:Topologie:IntCone0IsCone}, and $\WEff{f}{\s}{H}=\WEff{f}{\s}{J}$ because $\interior{J}=\interior{H}$. The second assertion of \Cref{th:CKonvImplik}, applied to $J$, shows that $f$ is weakly $\interior{J}$-quasiconvex on $\calD$. Hence \cite[Theorem~3.4]{GuntherKobisSchmollingTammer2023VectorialpenalisationvectoroptimisationreallineartopologicalspacesjournalArticle}, applied with the ordering cone $J$, yields $\x\in\WEff{f}{\calD}{J}=\WEff{f}{\calD}{H}$. By again using \Cref{th:PEff:PEffHenigEqui}~\eqref{gl:PEff:PEffHenigEqui:Eff=WEff}, we obtain $\x\in\PEffHenig{f}{\calD}{\c}$, proving \eqref{gl:MOP:Eff:PEff:IntCapPEff(S)InIntCapPEff(D)}.

		The first inclusion of \eqref{gl:SchrankePEffInt} is given by \eqref{gl:MOP:Eff:PEff:ScapPEffInPEff}, whereas the second one follows directly from \eqref{gl:MOP:Eff:PEff:IntCapPEff(S)InIntCapPEff(D)}. From this second inclusion we immediately get \eqref{gl:MOP:Eff:PEff:PEff/PEffInBD}. If $\s$ is open, we have $\s=\interior{\s}$ and $\s\cap\boundary[\meinUrbildraum]{\s}=\emptyset$, so that \eqref{gl:SchrankePEffInt} gives the last claim. If $\PEffHenig{f}{\s}{\c}\subseteq\interior[\meinUrbildraum]{\s}$, \eqref{gl:SchrankePEffInt} also immediately gives the last claim.
	\end{proof}

    If the whole space were admitted as a dilating cone, the local-to-global results from \Cref{th:MOP:Eff:PEff:LocalSolutionsInIntAreGlobalIfConvex} would fail.

    \begin{Beispiel}
        Let $\calD=\R$, $\s=(-1,1)$, and
		\[
			f(t)\coloneqq\bigl(\max\{|t|-1,0\},0\bigr).
		\]
		The function $f$ is constant on $\s$ and it is $C_{\mathrm{lex}}$-convex on $\calD$, because its first component is convex. If $\R^2$ were admitted as the only dilating cone, every point of $\s$ would be Henig-properly efficient on $\s$, whereas no point would be Henig-properly efficient on $\calD$, since $f$ is not constant there.
    \end{Beispiel}

	\section{Vectorial penalisation} \label{sec:vectorialPenalisation:PEff}

	This section contains the vectorial penalisation approach for general vector optimisation problems.
	Given a vector function $f: \calD \to Y$ (as considered in assumption \eqref{ass:VOPRestrVor}) and a scalar function $\nu: \calD \to \R$, we are interested in the extended (penalised) vector  function
	\begin{equation*}
		\fvb:\calD\to Y\times\R,\quad\fvb\coloneqq(f,\nu) \label{gl:vectorialPenalisation:PEff:problemFormulation:fvb}
	\end{equation*}
	and the extended (penalised) vector optimisation problem
	\begin{align}
		\argminKegel{\c\times\R_\geq}\limits_{x\in\calD}\ \fvb(x).\tag{$P^\vb_\calD$}\label{P:vectorialPenalisation:PEff:problemFormulation:PvbX}
	\end{align}
    Note that the underlying ordering cone $\c\times\R_\geq$ in the extended problem \eqref{P:vectorialPenalisation:PEff:problemFormulation:PvbX} is a nontrivial, pointed, convex cone in $Y \times \mathbb{R}$.

	In our approach, the function $\nu: \calD \to \R$ can be seen as a penalisation term w.r.t. the feasible set $\s$ of the vector optimisation problem \eqref{P:P}, which is given in \eqref{ass:VOPRestrVor}.
    Our aim is to show close relationships between the sets of Henig-properly efficient solutions of the constrained vector problem \eqref{P:P} and the unconstrained vector problems \eqref{P:PX} and \eqref{P:vectorialPenalisation:PEff:problemFormulation:PvbX}.

	Throughout this section, we consider the following assumptions:
	\begin{equation}
		\begin{cases}
			&\text{Assume } \eqref{ass:VOPRestrVor},\\
			&\nu:\calD\to\R,\\
			&\fvb\coloneqq(f,\nu).
		\end{cases} \tag{$A^\otimes$} \label{ass:vectorialPenalisation:PEff:problemFormulation:VOPRestrVBVor}
	\end{equation}

	\subsection{Penalisation functions}\label{sec:vectorialPenalisation:PEff:penalisationFunctions}

	Assume \eqref{ass:vectorialPenalisation:PEff:problemFormulation:VOPRestrVBVor} and consider a set $\calU\subseteq \calD$ with $\emptyset\neq \s\subseteq \calU$. For the penalisation function $\nu: \calD \to \R$,  we are interested in properties of the following type:
	\begin{align}
		&\forall x^0\in\boundary[\meinUrbildraum]{\s}:\ \lev_\leq\big(\calU,\nu,\nu(x^0)\big)=\s\tag{$\mathzapf A$1}\label{ass:vectorialPenalisation:PEff:penalisationFunctions:A1}\\
        &\forall x^0\in\boundary[\meinUrbildraum]{\s}:\ \lev_=\left(\calU,\nu,\nu(x^0)\right)=\boundary[\meinUrbildraum]{\s}\tag{$\mathzapf A$2}\label{ass:vectorialPenalisation:PEff:penalisationFunctions:A2}\\
		&\forall x^0\in \s:\ \lev_=\big(\calU,\nu,\nu(x^0)\big)=\lev_\leq\big(\calU,\nu,\nu(x^0)\big)=\s \tag{$\mathzapf{A}$3}\label{ass:vectorialPenalisation:PEff:penalisationFunctions:A3}
	\end{align}

	In the context of property \eqref{ass:vectorialPenalisation:PEff:penalisationFunctions:A1}, we additionally impose the condition $\boundary[\meinUrbildraum]{\s}\subseteq \calU$.

	For simplicity, if $\nu$ satisfies one of the properties \eqref{ass:vectorialPenalisation:PEff:penalisationFunctions:A1}, \eqref{ass:vectorialPenalisation:PEff:penalisationFunctions:A2}, or \eqref{ass:vectorialPenalisation:PEff:penalisationFunctions:A3} based on some sets $\calU$ and $\s$, we briefly use the notation
	\[\nu\in \mathzapf{A}_i(\s,\calU) \quad (\text{for } i\in\{1,2,3\}).\]

	We next collect some consequences of these assumptions relating the level sets of the penalisation function to the boundary and interior of the feasible set.

	\begin{Proposition}[{\cite[Proposition~4.1]{GuntherKobisSchmollingTammer2023VectorialpenalisationvectoroptimisationreallineartopologicalspacesjournalArticle}}] \label{th:vectorialPenalisation:PEff:penalisationFunctions:A1-A7Eigenschaften}
		Assume \eqref{ass:vectorialPenalisation:PEff:problemFormulation:VOPRestrVBVor} and let $\s\subseteq\calU\subseteq\calD$.
		\begin{enumerate}
			\item \label{item:vectorialPenalisation:PEff:penalisationFunctions:A1-A7Eigenschaften:A1algabges}
			Let $\boundary[\meinUrbildraum]{\s}\subseteq \calU$ and $\nu\in \mathzapf{A}_1(\s,\calU)$ for \eqref{ass:vectorialPenalisation:PEff:penalisationFunctions:A1}. Then $\s$ is closed and
			\begin{equation}
				\forall x^0\in\boundary[\meinUrbildraum]{\s}:\ \boundary[\meinUrbildraum]{\s}\subseteq\lev_=\big(\calU,\nu,\nu(x^0)\big) \label{ass:vectorialPenalisation:PEff:penalisationFunctions:A1-A7Eigenschaften:A1bdkonst}
			\end{equation}
			or equivalently
			\begin{equation}
				\forall x^0\in\boundary[\meinUrbildraum]{\s}:\ \lev_<\big(\calU,\nu,\nu(x^0)\big)\subseteq\interior[\meinUrbildraum]{\s}. \label{ass:vectorialPenalisation:PEff:penalisationFunctions:A1-A7Eigenschaften:A1lev<InIcore}
			\end{equation}

            \item \label{item:vectorialPenalisation:PEff:penalisationFunctions:A1-A7Eigenschaften:A1A2lev<=Icore}
		Let $\boundary{\s}\subseteq \calU$ and assume $\nu$ satisfies \eqref{ass:vectorialPenalisation:PEff:penalisationFunctions:A1} and \eqref{ass:vectorialPenalisation:PEff:penalisationFunctions:A2} with $\nu\in \mathzapf{A}_1(\s,\calU)\cap \mathzapf{A}_2(\s,\calU)$. Then it is
    		\begin{equation}
		      \forall x^0\in\boundary{\s}:\ \lev_<\big(\calU,\nu,\nu(x^0)\big)=\interior{\s}. \label{ass:vectorialPenalisation:PEff:penalisationFunctions:A1-A7Eigenschaften:A1A2lev<=Icore}
    		\end{equation}

			\item \label{item:vectorialPenalisation:PEff:penalisationFunctions:A1-A7Eigenschaften:A3lev<=leer}
			If $\nu\in \mathzapf{A}_3(\s,\calU)$, then
			\begin{equation}
				\forall x^0\in \s:\ \lev_<\big(\calU,\nu,\nu(x^0)\big)=\leereMenge. \label{ass:vectorialPenalisation:PEff:penalisationFunctions:A1-A7Eigenschaften:A3lev<leer}
			\end{equation}
		\end{enumerate}
	\end{Proposition}

    \begin{Bemerkung}
        Several types of penalisation functions that satisfy assumptions \textup{[}\eqref{ass:vectorialPenalisation:PEff:penalisationFunctions:A1} and \eqref{ass:vectorialPenalisation:PEff:penalisationFunctions:A2}\textup{]} and \eqref{ass:vectorialPenalisation:PEff:penalisationFunctions:A3}, respectively, are known. All these penalisation functions can be used within our vectorial penalisation approach; see \cite{GuntherKobisSchmollingTammer2023VectorialpenalisationvectoroptimisationreallineartopologicalspacesjournalArticle} for an overview.
        Property \eqref{ass:vectorialPenalisation:PEff:penalisationFunctions:A2} can be viewed as a refinement of \eqref{ass:vectorialPenalisation:PEff:penalisationFunctions:A1}. Indeed, by \Cref{th:vectorialPenalisation:PEff:penalisationFunctions:A1-A7Eigenschaften}~\ref{item:vectorialPenalisation:PEff:penalisationFunctions:A1-A7Eigenschaften:A1A2lev<=Icore}, the combination of \eqref{ass:vectorialPenalisation:PEff:penalisationFunctions:A1} and \eqref{ass:vectorialPenalisation:PEff:penalisationFunctions:A2} yields the identity \eqref{ass:vectorialPenalisation:PEff:penalisationFunctions:A1-A7Eigenschaften:A1A2lev<=Icore}, so boundary points are completely separated from the interior of $\s$ by the penalisation function. However, this additional separation does not seem to yield stronger results for properly efficient points in the contexts considered below. For our later arguments, the decisive ingredients are the ordering information from \eqref{ass:vectorialPenalisation:PEff:penalisationFunctions:A1} and the stronger constancy property encoded in \eqref{ass:vectorialPenalisation:PEff:penalisationFunctions:A3}. Therefore, we keep \eqref{ass:vectorialPenalisation:PEff:penalisationFunctions:A2} as a natural auxiliary assumption and include it for the sake of completeness, but from this point on we restrict our attention to penalisation functions satisfying \eqref{ass:vectorialPenalisation:PEff:penalisationFunctions:A1} or \eqref{ass:vectorialPenalisation:PEff:penalisationFunctions:A3}.
    \end{Bemerkung}

	\subsection{The effects of adding/removing a penalisation term} \label{sec:vectorialPenalisation:PEff:AddRemove}

	In this section, we analyse the effects of adding/removing a penalisation term to/from a vector-valued objective function on the nature of the sets of Henig-properly efficient solutions of the vector optimisation problems.
    
    The purpose of this intermediate step is to isolate the precise mechanism by which the penalisation term changes proper efficiency. Before turning to the complete characterisation of the constrained problem, we first clarify in which situations properly efficient points are preserved and in which situations new properly efficient points may appear or disappear after penalisation. This prepares the main result conceptually: the final decomposition of the constrained problem is obtained by combining the points that remain visible in the original objective with those that are detected only after adding the penalisation term.

	Consider the (non-penalised) vector optimisation problem \eqref{P:P} under the assumption \eqref{ass:VOPRestrVor}. We are interested in the unconstrained extended (penalised) vector optimisation problem
	\begin{align}
		\argminKegel{\c\times\R_\geq}\limits_{x\in\s}\ \fvb(x).\tag{$P^\vb_\s$}\label{P:vectorialPenalisation:PEff:AddRemove:PvbS}
	\end{align}

	We proceed in two steps. First, we introduce a structured class of dilating cones for the ordering cone $\c\times\R_\geq$ of \eqref{P:vectorialPenalisation:PEff:AddRemove:PvbS} and study its properties and its relation to the full class of dilating cones. In the following part, we use these cone constructions to derive relationships between the Henig-properly efficient solution sets of the original and penalised problems.

    \subsubsection{Construction and properties of a structured class of dilating cones for \eqref{P:vectorialPenalisation:PEff:AddRemove:PvbS}}\label{sec:vectorialPenalisation:PEff:AddRemove:structuredDilatingCones}

	First we need some relationship between the Henig dilating cones of the ordering cone $\c$ of \eqref{P:P} and the ones of $\c\times\R_\geq$ of \eqref{P:vectorialPenalisation:PEff:AddRemove:PvbS}. For this we construct from a dilating cone for \eqref{P:P} a new one for \eqref{P:vectorialPenalisation:PEff:AddRemove:PvbS}, while for the other direction we simply take a suitable projection.

	\begin{Definition} \label{def:vectorialPenalisation:PEff:AddRemove:structuredDilatingCones:HenigDilatingConesVB}
		Assume \eqref{ass:VOPRestrVor}. For any $H\in\HenigDilatingCones{\c}$ and $h\in\interior{H}$ we define
        \begin{equation*}
            H^\vb_{h}\coloneqq\Menge{(v,t)\in\meinBildraum\times\R}{v\in H-th}
        \end{equation*}
		and
		\[\HenigDilatingConesVB{\c}\coloneqq\Menge{H^\vb_{h}}{H\in\HenigDilatingCones{\c},\ h\in\interior{H}}.\]
	\end{Definition}
    
	\begin{Bemerkung}
        The construction in \Cref{def:vectorialPenalisation:PEff:AddRemove:structuredDilatingCones:HenigDilatingConesVB} admits a simple geometric interpretation. For every $t\in\R$, the horizontal section of $H^\vb_h$ at height $t$ is given by
        \[
            \SetRestrictionToPoint{H^\vb_h}{(\meinBildraum,t)}=H-th.
        \]
        Hence, $H^\vb_h$ is obtained by stacking translated copies of $H$ along the $\R$-direction, where the translation is governed by the vector $h\in\interior{H}$. In particular, the section at height $0$ is precisely $H$, while for $t<0$ the sections move further into the interior of $H$. Indeed, for every nontrivial solid convex cone $H$ and every $h\in\interior{H}$ one has
        \[
            \interior{H}=H+\R_>h
            \qquad\text{and}\qquad
            \meinBildraum=H-\R_>h,
        \]
        see \cite[Lemma~2.4]{GuentherPopoviciOrzan2025}. Consequently, as $t>0$ increases, the sections $H-th$ eventually sweep through the whole space, whereas for $t<0$ they remain inside $\interior{H}$.
        
        At the same time, the geometric description reveals a limitation of the construction. By \Cref{th:PEff:PEffHenigEqui}~\eqref{gl:PEff:PEffHenigEqui:OnlyOpenH}, it suffices for the original problem to consider dilating cones satisfying $H=\interior{H}\cup\{0_\meinBildraum\}$.
        Even if $H$ has this property, the associated cone $H_h^\vb$ does not satisfy its product-space analogue: it contains the entire line
        \[
            \Menge{(-th,t)}{t\in\R},
        \]
        which is disjoint from $\interior{H_h^\vb}$ (see \Cref{th:Topologie:IntCone0IsCone}). Consequently, $H_h^\vb\setminus\{(0_\meinBildraum,0)\}$ is not open. Thus, although the construction yields dilating cones for the penalised problem, it does not preserve the useful property available for the original problem. In particular, each $H_h^\vb$ is non-pointed.
    \end{Bemerkung}


	Given this newly constructed set we need to show some properties, that will later help us to compare the original and penalised problems. Mainly, we need to show that $\HenigDilatingConesVB{\c}$ is indeed a family of dilating cones for $\c\times\R_\geq$. By a projection to the section at height zero we easily get back the cone one from $\HenigDilatingConesVB{\c}$ was constructed from, showing that there are at least as many dilating cones for $\c\times\R_\geq$ as there are for $\c$.

    \begin{Lemma}\label{th:vectorialPenalisation:PEff:AddRemove:structuredDilatingCones:H^ocross_hProperties}
        Assume \eqref{ass:VOPRestrVor}. Then the following assertions hold:
        \begin{enumerate}
            \item \label{item:vectorialPenalisation:PEff:AddRemove:structuredDilatingCones:H^ocross_hProperties:NegativeHalfspace}
            For every $H_h^\vb\in\HenigDilatingConesVB{\c}$,
            \begin{align}
                H\times\R_\geq
                &\subseteq H_h^\vb\cap(\meinBildraum\times\R_\geq),
                \label{gl:vectorialPenalisation:PEff:AddRemove:structuredDilatingCones:H^ocross_hProperties:PositiveHalfspace}\\
                \interior[\meinBildraum]{H}\times\R_\geq
                &\subseteq\interior[\meinBildraum\times\R]{H_h^\vb}\cap(\meinBildraum\times\R_\geq),
                \label{gl:vectorialPenalisation:PEff:AddRemove:structuredDilatingCones:H^ocross_hProperties:InteriorPositiveHalfspace}\\
                H_h^\vb\cap(\meinBildraum\times\R_\leq)
                &\subseteq H\times\R_\leq,
                \label{gl:vectorialPenalisation:PEff:AddRemove:structuredDilatingCones:H^ocross_hProperties:NegativeHalfspace}\\
                \interior[\meinBildraum\times\R]{H_h^\vb}\cap(\meinBildraum\times\R_\leq)
                &\subseteq\interior[\meinBildraum]{H}\times\R_\leq.
                \label{gl:vectorialPenalisation:PEff:AddRemove:structuredDilatingCones:H^ocross_hProperties:InteriorNegativeHalfspace}
            \end{align}

            \item \label{item:vectorialPenalisation:PEff:AddRemove:structuredDilatingCones:H^ocross_hProperties:SetRestrictionH^ocross_h=H}
    		For any $H_h^\vb\in\HenigDilatingConesVB{\c}$ it is $\SetRestrictionToPoint{H^\vb_{h}}{(\meinBildraum,0)}=H$.

		\item \label{item:vectorialPenalisation:PEff:AddRemove:structuredDilatingCones:H^ocross_hProperties:H^ocross_hINHeDC(CxR)}
            $\HenigDilatingConesVB{\c}\subseteq\HenigDilatingCones{\c\times\R_\geq}.$

            \item \label{item:vectorialPenalisation:PEff:AddRemove:structuredDilatingCones:H^ocross_hProperties:MengeProjHx=HenigDC(C)}
            $\Menge{\SetRestrictionToPoint{H^\vb}{(\meinBildraum,0)}}{H^\vb\in\HenigDilatingCones{\c\times\R_\geq}}=\HenigDilatingCones{\c}.$
		\end{enumerate}
    \end{Lemma}
    \begin{proof}
        \begin{enumerate}[leftmargin=*]
            \item[\ref{item:vectorialPenalisation:PEff:AddRemove:structuredDilatingCones:H^ocross_hProperties:NegativeHalfspace}]
            Define
            \[
                \Phi:\meinBildraum\times\R\to\meinBildraum\times\R,
                \qquad \Phi(v,t)\coloneqq(v+th,t).
            \]
            By \Cref{th:Topology:ProductSpace:x+thContinuous} the mapping $\Phi$ is a linear homeomorphism with inverse $\Phi^{-1}(w,t)=(w-th,t)$, and the definition of $H_h^\vb$ gives
            \[
                H_h^\vb=\Phi^{-1}(H\times\R).
            \]
            Consequently,
            \begin{align}
                \interior[\meinBildraum\times\R]{H_h^\vb}
                &=\Phi^{-1}\left(\interior[\meinBildraum\times\R]{H\times\R}\right) \notag \\
                &=\Phi^{-1}\left(\interior[\meinBildraum]{H}\times\R\right) \notag \\
                &=\Menge{(v,t)\in\meinBildraum\times\R}{v\in\interior[\meinBildraum]{H}-th}. \label{gl:vectorialPenalisation:PEff:AddRemove:structuredDilatingCones:HeDC(C)HeDC(CxR):Int(Hhh)}
            \end{align}

            First let $t\geq0$. For $v\in H$, since $H$ is a convex cone, it is $th\in H$, and by \Cref{th:Mengeneigenschaften:Kegel:konvexÄquivalenz} $H+H=H$, giving $v+th\in H$, and hence $(v,t)\in H_h^\vb$. If $v\in\interior[\meinBildraum]{H}$, then, again by \Cref{th:Mengeneigenschaften:Kegel:konvexÄquivalenz},
            \[
                v+th\in\interior[\meinBildraum]{H}+H
                =\interior[\meinBildraum]{H},
            \]
            and \eqref{gl:vectorialPenalisation:PEff:AddRemove:structuredDilatingCones:HeDC(C)HeDC(CxR):Int(Hhh)} yields $(v,t)\in\interior[\meinBildraum\times\R]{H_h^\vb}$. This proves \eqref{gl:vectorialPenalisation:PEff:AddRemove:structuredDilatingCones:H^ocross_hProperties:PositiveHalfspace} and \eqref{gl:vectorialPenalisation:PEff:AddRemove:structuredDilatingCones:H^ocross_hProperties:InteriorPositiveHalfspace}.

            Now let $t\leq0$. If $(v,t)\in H_h^\vb$, then $v+th\in H$ and $-th\in H$, so
            \[
                v=(v+th)+(-th)\in H+H=H.
            \]
            If $(v,t)\in\interior[\meinBildraum\times\R]{H_h^\vb}$, then $v+th\in\interior[\meinBildraum]{H}$ and hence, by \Cref{th:Mengeneigenschaften:Kegel:konvexÄquivalenz},
            \[
                v=(v+th)+(-th)
                \in\interior[\meinBildraum]{H}+H
                =\interior[\meinBildraum]{H}.
            \]
            This proves \eqref{gl:vectorialPenalisation:PEff:AddRemove:structuredDilatingCones:H^ocross_hProperties:NegativeHalfspace} and \eqref{gl:vectorialPenalisation:PEff:AddRemove:structuredDilatingCones:H^ocross_hProperties:InteriorNegativeHalfspace}.
    
            \item[\ref{item:vectorialPenalisation:PEff:AddRemove:structuredDilatingCones:H^ocross_hProperties:SetRestrictionH^ocross_h=H}]
            It is $H^\vb_{h}\cap\left(\meinBildraum\times\{0\}\right)=H\times\{0\}$, so that $\SetRestrictionToPoint{H^\vb_{h}}{(\meinBildraum,0)}=\ProjectionSet{\meinBildraum}{H^\vb_{h}\cap\left(\meinBildraum\times\{0\}\right)}=H$.

            \item[\ref{item:vectorialPenalisation:PEff:AddRemove:structuredDilatingCones:H^ocross_hProperties:H^ocross_hINHeDC(CxR)}]
            Consider any $H^\vb_h\in\HenigDilatingConesVB{\c}$. For any $(v,t)\in H^\vb_{h}$ and any $\lambda\in\R_\geq$ it is $v\in H-th$ and therefore since $H$ is a cone $\lambda(v+th)\in H$, resulting in $\lambda v\in H-\lambda th$, so that $\lambda(v,t)\in H^\vb_{h}$. For any $(v^1,t_1),(v^2,t_2)\in H^\vb_{h}$ and $\lambda\in(0,1)$ we have $v^1+t_1h,v^2+t_2h\in H$, where the convexity of $H$ gives $\lambda(v^1+t_1h)+(1-\lambda)(v^2+t_2h)\in H$ so that $\lambda v^1+(1-\lambda)v^2\in H-\left(\lambda t_1+(1-\lambda)t_2\right)h$ and so $\lambda(v^1,t_1)+(1-\lambda)(v^2,t_2)\in H^\vb_{h}$. Thus, $H^\vb_{h}$ is a convex cone.\\
            Let $(c,t)\in(\c\times\R_\geq)\setminus\{(0_\meinBildraum,0)\}$. If $c\neq0_\meinBildraum$, then $c\in\interior[\meinBildraum]{H}$, and \eqref{gl:vectorialPenalisation:PEff:AddRemove:structuredDilatingCones:H^ocross_hProperties:InteriorPositiveHalfspace} gives $(c,t)\in\interior[\meinBildraum\times\R]{H_h^\vb}$. If $c=0_\meinBildraum$, then $t>0$, and \Cref{th:Topologie:IntCone0IsCone} yields $th\in\interior[\meinBildraum]{H}$. Hence \eqref{gl:vectorialPenalisation:PEff:AddRemove:structuredDilatingCones:HeDC(C)HeDC(CxR):Int(Hhh)} gives $(0_\meinBildraum,t)\in\interior[\meinBildraum\times\R]{H_h^\vb}$. Therefore,
            \[
                (\c\times\R_\geq)\setminus\{(0_\meinBildraum,0)\}
                \subseteq\interior[\meinBildraum\times\R]{H_h^\vb}.
            \]
            Moreover, $H_h^\vb\neq\meinBildraum\times\R$, since by part~\ref{item:vectorialPenalisation:PEff:AddRemove:structuredDilatingCones:H^ocross_hProperties:SetRestrictionH^ocross_h=H} it is $\SetRestrictionToPoint{H^\vb_{h}}{(\meinBildraum,0)}=H\subset \meinBildraum$. Hence $H^\vb_{h}\in\HenigDilatingCones{\c\times\R_\geq}$.

    
            \item[\ref{item:vectorialPenalisation:PEff:AddRemove:structuredDilatingCones:H^ocross_hProperties:MengeProjHx=HenigDC(C)}] For any $H^\vb\in\HenigDilatingCones{\c\times\R_\geq}$ it is $H^\vb\cap\left(\meinBildraum\times\{0\}\right)$ as intersection of convex cones again a convex cone, so that also $\SetRestrictionToPoint{H^\vb}{(\meinBildraum,0)}=\ProjectionSet{\meinBildraum}{H^\vb\cap\left(\meinBildraum\times\{0\}\right)}$ is a convex cone. From $H^\vb\in\HenigDilatingCones{\c\times\R_\geq}$ we get $\left(\c\ohne\{0_\meinBildraum\}\right)\times\{0\}\subseteq\interior[\meinBildraum\times\R]{H^\vb}\cap(\meinBildraum\times\{0\})$. By \Cref{th:Topologie:ProductTopology:RestrictionOfInteriorOfSet=InteriorOfRestriction} this gives
    		\[\c\ohne\{0\}\subseteq\SetRestrictionToPoint{\interior[\meinBildraum\times\R]{H^\vb}}{(\meinBildraum,0)}=\interior[\meinBildraum]{\SetRestrictionToPoint{H^\vb}{(\meinBildraum,0)}}\]
		    If $\SetRestrictionToPoint{H^\vb}{(\meinBildraum,0)}=\meinBildraum$, then $(0_\meinBildraum,0)$ belongs to $\interior[\meinBildraum\times\R]{H^\vb}$ by \Cref{th:Topologie:ProductTopology:RestrictionOfInteriorOfSet=InteriorOfRestriction}, and hence $H^\vb=\meinBildraum\times\R$, a contradiction. Thus the section is a proper subset of $\meinBildraum$, and consequently $\SetRestrictionToPoint{H^\vb}{(\meinBildraum,0)}\in\HenigDilatingCones{\c}$, giving the first inclusion in \ref{item:vectorialPenalisation:PEff:AddRemove:structuredDilatingCones:H^ocross_hProperties:MengeProjHx=HenigDC(C)}. \par
            The reverse inclusion for concluding \ref{item:vectorialPenalisation:PEff:AddRemove:structuredDilatingCones:H^ocross_hProperties:MengeProjHx=HenigDC(C)} follows by parts~\ref{item:vectorialPenalisation:PEff:AddRemove:structuredDilatingCones:H^ocross_hProperties:H^ocross_hINHeDC(CxR)} and \ref{item:vectorialPenalisation:PEff:AddRemove:structuredDilatingCones:H^ocross_hProperties:SetRestrictionH^ocross_h=H} giving us
            \[\HenigDilatingCones{\c}=\Menge{\SetRestrictionToPoint{H^\vb_{h}}{(\meinBildraum,0)}}{H^\vb_h\in\HenigDilatingConesVB{\c}}\subseteq \Menge{\SetRestrictionToPoint{H^\vb}{(\meinBildraum,0)}}{H^\vb\in\HenigDilatingCones{\c\times\R_\geq}}.\qedhere\]
        \end{enumerate}
    \end{proof}

	Although \Cref{th:vectorialPenalisation:PEff:AddRemove:structuredDilatingCones:H^ocross_hProperties}~\ref{item:vectorialPenalisation:PEff:AddRemove:structuredDilatingCones:H^ocross_hProperties:MengeProjHx=HenigDC(C)} shows that the zero-level sections of dilating cones for the penalised problem exhaust $\HenigDilatingCones{\c}$, the structured family $\HenigDilatingConesVB{\c}$ need not exhaust $\HenigDilatingCones{\c\times\R_\geq}$.
	
	As an immediate consequence of the previous results we get the following corollary.

	\begin{Korollar}\label{th:vectorialPenalisation:PEff:AddRemove:structuredDilatingCones:NonemptyDilatingConesProductEquivalence}
		The family of Henig dilating cones of $\c$ is nonempty if and only if the family of Henig dilating cones of $\c\times\R_\geq$ is nonempty, that is,
		\begin{equation}
			\HenigDilatingCones{\c}\neq\emptyset
			\quad\Longleftrightarrow\quad
			\HenigDilatingCones{\c\times\R_\geq}\neq\emptyset.
			\label{gl:vectorialPenalisation:PEff:AddRemove:structuredDilatingCones:NonemptyDilatingConesProductEquivalence}
		\end{equation}
	\end{Korollar}
	\begin{proof}
		If $H\in\HenigDilatingCones{\c}$, choose any $h\in\interior[\meinBildraum]{H}$. Then
		\Cref{th:vectorialPenalisation:PEff:AddRemove:structuredDilatingCones:H^ocross_hProperties}~\ref{item:vectorialPenalisation:PEff:AddRemove:structuredDilatingCones:H^ocross_hProperties:H^ocross_hINHeDC(CxR)}
		gives $H_h^\vb\in\HenigDilatingCones{\c\times\R_\geq}$. Conversely, if
		$H^\vb\in\HenigDilatingCones{\c\times\R_\geq}$, then
		\Cref{th:vectorialPenalisation:PEff:AddRemove:structuredDilatingCones:H^ocross_hProperties}~\ref{item:vectorialPenalisation:PEff:AddRemove:structuredDilatingCones:H^ocross_hProperties:MengeProjHx=HenigDC(C)}
		yields $\SetRestrictionToPoint{H^\vb}{(\meinBildraum,0)}\in\HenigDilatingCones{\c}$.
	\end{proof}

	We demonstrate the results of \Cref{th:vectorialPenalisation:PEff:AddRemove:structuredDilatingCones:H^ocross_hProperties} in the following example, where we use Bishop-Phelps cones and their properties (see, e.g., \cite{HaJahn2023}, \cite{GuntherKhazayelStrugariuTammer2024BishopPhelpsTypeScalarizationVectorOptimizationRealTopologicalLinearSpacespreprint} and the references therein) to immediately get representations of the interior of these cones.
    
	\begin{Beispiel}
		We consider the cone
		\[H^\vb=\Menge{(x,y,z)\in\R^3}{\frac{2}{\sqrt3}\skp{(1,1,1)}{(x,y,z)}\geq\norm{(x,y,z)}_{l^2}}.\]
		with the by $\ell^2(\R^2)$ induced topology and $\c=\R^2_\geq$.

		We have $H^\vb\in\HenigDilatingCones{\c\times\R_\geq}=\HenigDilatingCones{\R^3_\geq}$, because for every $(x,y,z)\in\R^3_\gneq$,
		\[\norm{(x,y,z)}_{\ell^2}^2=x^2+y^2+z^2\leq(x+y+z)^2<\frac{4}{3}(x+y+z)^2=\left(\frac2{\sqrt3}\skp{(1,1,1)}{(x,y,z)}\right)^2,\]
		which by \cite[Lemma~2.12]{GuntherKhazayelStrugariuTammer2024BishopPhelpsTypeScalarizationVectorOptimizationRealTopologicalLinearSpacespreprint} implies $\R^3_\gneq\subseteq\interior[\R^3]{H^\vb}$.

		As in \Cref{th:vectorialPenalisation:PEff:AddRemove:structuredDilatingCones:H^ocross_hProperties} we want to take a closer look at $\SetRestrictionToPoint{H^\vb}{(\R^2,0)}$ and get
		\begin{align*}
			H&\coloneqq\SetRestrictionToPoint{H^\vb}{(\R^2,0)}=\ProjectionSet{\R^2}{H^\vb\cap\left(\R^2\times\{0\}\right)} \\
			&=\Menge{(x,y)\in\R^2}{\frac{2}{\sqrt{3}}\skp{(1,1)}{(x,y)}\geq\norm{(x,y)}_{\ell^2}}.
		\end{align*}
		It is $H\in\HenigDilatingCones{\c}$, because for every $(x,y)\in\R^2_\gneq$ it is $x^2+y^2\leq(x+y)^2<\frac43(x+y)^2$ as before, indicating $\R^2_\gneq\subseteq\interior[\R^2]{H}$.

		Considering now $H^\vb_h$ for $h=(1,1)$ we have
		\begin{align*}
			H^\vb_h&=\Menge{(x,y,t)\in\R^3}{(x,y)\in H-th} \\
			&=\Menge{(x,y,t)\in\R^3}{\frac{2}{\sqrt3}\skp{(1,1)}{(x+t,y+t)}\geq\norm{(x+t,y+t)}_{\ell^2}}.
		\end{align*}
        Obviously, it is $\SetRestrictionToPoint{H^\vb_h}{(\R^2,0)}=H$, and with the same argument as before $\R^3_\gneq\subseteq\interior[\R^3]{H^\vb_h}$. Since $H\neq\R^2$, this also implies $H^\vb_h\neq\R^3$. Consequently, $H^\vb_h\in\HenigDilatingCones{\R^3_\geq}$.
	\end{Beispiel}

For the converse direction, we need further properties of arbitrary dilating cones in $\HenigDilatingCones{\c\times\R_\geq}$. More precisely, we relate such a cone to the product of its zero-level section with $\R_\geq$.
The positive-half-space inclusions below are the general counterparts of the corresponding inclusions for the structured product-space dilating cones in \Cref{th:vectorialPenalisation:PEff:AddRemove:structuredDilatingCones:H^ocross_hProperties}~\ref{item:vectorialPenalisation:PEff:AddRemove:structuredDilatingCones:H^ocross_hProperties:NegativeHalfspace}.

	\begin{Proposition}\label{th:vectorialPenalisation:PEff:AddRemove:structuredDilatingCones:HxRInJ}
		Assume \eqref{ass:VOPRestrVor}. Let $H^\vb\in\HenigDilatingCones{\c\times\R_\geq}$. Then 
		\begin{align}
			\SetRestrictionToPoint{H^\vb}{(\meinBildraum,0)}\times\R_\geq&\subseteq H^\vb, \label{gl:vectorialPenalisation:PEff:AddRemove:structuredDilatingCones:HxRInJ}\\
            \SetRestrictionToPoint{H^\vb}{(\meinBildraum,0)}\times\R_>&\subseteq\interior[\meinBildraum\times\R]{H^\vb}. \label{gl:vectorialPenalisation:PEff:AddRemove:structuredDilatingCones:HxR>InIntJ} \\
			\interior[\meinBildraum]{\SetRestrictionToPoint{H^\vb}{(\meinBildraum,0)}}\times\R_\geq&\subseteq\interior[\meinBildraum\times\R]{H^\vb}. \label{gl:vectorialPenalisation:PEff:AddRemove:structuredDilatingCones:IntHxRInIntJ}
		\end{align}
	\end{Proposition}
	\begin{proof}
		For any $(x,t)\in \SetRestrictionToPoint{H^\vb}{(\meinBildraum,0)}\times\R_\geq$ if $t=0$ we immediately get $(x,t)\in H^\vb$. So from now on we assume $t>0$. Take any $s>t$. Since $0_\meinBildraum\in\c$ and $H^\vb\in\HenigDilatingCones{\c\times\R_\geq}$ we get $(0_\meinBildraum,s)\in H^\vb$. Because $\SetRestrictionToPoint{H^\vb}{(\meinBildraum,0)}$ is a cone and $x\in \SetRestrictionToPoint{H^\vb}{(\meinBildraum,0)}$ we also get $\frac{s}{s-t}x\in \SetRestrictionToPoint{H^\vb}{(\meinBildraum,0)}$ and thus $\left(\frac{s}{s-t}x,0\right)\in H^\vb$. By convexity of $H^\vb$ since $H^\vb\in\HenigDilatingCones{\c\times\R_\geq}$ and with \Cref{th:Mengeneigenschaften:Kegel:konvexÄquivalenz} we now get
		\[(x,t)=\frac{t}{s}(0_\meinBildraum,s)+\left(1-\frac{t}{s}\right)\left(\frac{s}{s-t}x,0\right)\in H^\vb.\]
		This proves that $\SetRestrictionToPoint{H^\vb}{(\meinBildraum,0)}\times\R_\geq\subseteq H^\vb$.

		For any $x\in\SetRestrictionToPoint{H^\vb}{(\meinBildraum,0)}$ and $t>0$, we have $(x,0)\in H^\vb$ and
		\[
			(0_\meinBildraum,t)\in(\c\times\R_\geq)\setminus\{(0_\meinBildraum,0)\}\subseteq\interior[\meinBildraum\times\R]{H^\vb}.
		\]
		Since $H^\vb$ is a convex cone, by \Cref{th:Mengeneigenschaften:Kegel:konvexÄquivalenz} it satisfies
		$H^\vb+\interior[\meinBildraum\times\R]{H^\vb}\subseteq\interior[\meinBildraum\times\R]{H^\vb}$.
		Therefore,
		\[
			(x,t)=(x,0)+(0_\meinBildraum,t)\in\interior[\meinBildraum\times\R]{H^\vb},
		\]
		which proves \eqref{gl:vectorialPenalisation:PEff:AddRemove:structuredDilatingCones:HxR>InIntJ}.
        
        By \eqref{gl:vectorialPenalisation:PEff:AddRemove:structuredDilatingCones:HxR>InIntJ}, we already have $\interior[\meinBildraum]{\SetRestrictionToPoint{H^\vb}{(\meinBildraum,0)}}\times\R_>\subseteq\interior[\meinBildraum\times\R]{H^\vb}$. It remains to show that
		\[
			\interior[\meinBildraum]{\SetRestrictionToPoint{H^\vb}{(\meinBildraum,0)}}\times\{0\}\subseteq\interior[\meinBildraum\times\R]{H^\vb}.
		\]
		By assumption \eqref{ass:VOPRestrVor}, the cone $\c$ is nontrivial, so there exists $c\in\c\setminus\{0_\meinBildraum\}$. Since $H^\vb\in\HenigDilatingCones{\c\times\R_\geq}$, we have
		\[
			(c,0)\in\bigl(\c\times\R_\geq\bigr)\setminus\{(0_\meinBildraum,0)\}\subseteq\interior[\meinBildraum\times\R]{H^\vb}.
		\]
		Thus, by \Cref{th:Topologie:ProductTopology:RestrictionOfInteriorOfSet=InteriorOfRestriction},
		\begin{equation*}
			\interior[\meinBildraum]{\SetRestrictionToPoint{H^\vb}{(\meinBildraum,0)}}\times\{0\}=\SetRestrictionToPoint{\interior[\meinBildraum\times\R]{H^\vb}}{(\meinBildraum,0)}\times\{0\}\subseteq\interior[\meinBildraum\times\R]{H^\vb}. \label{gl:vectorialPenalisation:PEff:AddRemove:structuredDilatingCones:Int(JRestriction)x0=IntHRestrictionx0InIntJ}
		\end{equation*}
		This proves \eqref{gl:vectorialPenalisation:PEff:AddRemove:structuredDilatingCones:IntHxRInIntJ}.
        %
		%
	\end{proof}

    The preceding proposition shows that every arbitrary dilating cone contains the product of its zero-level section with $\R_\geq$.

	From the preceding proposition we easily obtain the following exclusion relations.


	\begin{Korollar}\label{th:vectorialPenalisation:PEff:AddRemove:structuredDilatingCones:HxRExclusions}
		Assume \eqref{ass:VOPRestrVor}. Let $H^\vb\in\HenigDilatingCones{\c\times\R_\geq}$ and set $H_0\coloneqq\SetRestrictionToPoint{H^\vb}{(\meinBildraum,0)}$. Then
		\begin{align}
			\left((-H_0)\times\R_\leq\right)\cap H^\vb&=\left(H_0\cap(-H_0)\right)\times\{0\},\label{gl:vectorialPenalisation:PEff:AddRemove:structuredDilatingCones:HxRExclusions:negH0xRleqCapHvbEqH0capNegH0xZero}\\
			\left((-H_0)\times\R_<\right)\cap H^\vb&=\emptyset,\label{gl:vectorialPenalisation:PEff:AddRemove:structuredDilatingCones:HxRExclusions:negH0xRltCapHvbEmpty}\\
			\left((-\interior[\meinBildraum]{H_0})\times\R_\leq\right)\cap H^\vb&=\emptyset.\label{gl:vectorialPenalisation:PEff:AddRemove:structuredDilatingCones:HxRExclusions:negInteriorH0xRleqCapHvbEmpty}
		\end{align}
	\end{Korollar}
	\begin{proof}
		\Cref{th:Mengeneigenschaften:Kegel:InteriorNegativeConeDisjoint} implies
		\[
			H^\vb\cap\bigl(-\interior[\meinBildraum\times\R]{H^\vb}\bigr)=\emptyset.
		\]
		By \Cref{th:vectorialPenalisation:PEff:AddRemove:structuredDilatingCones:HxRInJ} we thus get \eqref{gl:vectorialPenalisation:PEff:AddRemove:structuredDilatingCones:HxRExclusions:negH0xRltCapHvbEmpty} and \eqref{gl:vectorialPenalisation:PEff:AddRemove:structuredDilatingCones:HxRExclusions:negInteriorH0xRleqCapHvbEmpty}, because
		\begin{align*}
			\left((-H_0)\times\R_<\right)\cap H^\vb
			&\subseteq\left(-\interior[\meinBildraum\times\R]{H^\vb}\right)\cap H^\vb=\emptyset, \\
			\left((-\interior[\meinBildraum]{H_0})\times\R_\leq\right)\cap H^\vb
			&\subseteq\left(-\interior[\meinBildraum\times\R]{H^\vb}\right)\cap H^\vb=\emptyset.
		\end{align*}
		From \eqref{gl:vectorialPenalisation:PEff:AddRemove:structuredDilatingCones:HxRExclusions:negH0xRltCapHvbEmpty} we get
		\begin{align*}
			\left((-H_0)\times\R_\leq\right)\cap H^\vb
			&=\left(\left((-H_0)\times\R_<\right)\cap H^\vb\right)\cup\left(\left((-H_0)\times\{0\}\right)\cap H^\vb\right) \\
			&=\left((-H_0)\cap H_0\right)\times\{0\}. \qedhere
		\end{align*}
	\end{proof}



    \subsubsection{Proper efficiency under penalisation}\label{sec:vectorialPenalisation:PEff:AddRemove:PenUnpen}

	Equipped with these structural results, we now study the relationships between the Henig-properly efficient points of \eqref{P:P} and \eqref{P:vectorialPenalisation:PEff:AddRemove:PvbS}. We first examine how Henig-proper efficiency may be created by adding the penalisation component. This direction is most naturally treated with an arbitrary cone $H^\vb\in\HenigDilatingCones{\c\times\R_\geq}$ and its zero-level section in $\meinBildraum$.
	We then investigate the reverse relation, namely the possible loss of Henig-proper efficiency after adding the penalisation component. For this more delicate direction, the structured cones $H_h^\vb\in\HenigDilatingConesVB{\c}$ allow us to track how the penalisation value affects the ordering of improvements and to characterise when a strict improvement is accompanied by a decrease in that value, causing a loss of Henig-proper efficiency.

	Although $\s$ denotes the common feasible set throughout this section, we want to point out, that $\s$ can be any arbitrary set with $\s\subseteq\calD$, so that we later on may use the results also for $\calD$.

    The preceding inclusions \eqref{gl:vectorialPenalisation:PEff:AddRemove:structuredDilatingCones:HxRInJ} and \eqref{gl:vectorialPenalisation:PEff:AddRemove:structuredDilatingCones:IntHxRInIntJ} allow us to relate improvements in the original objective to improvements in the penalised objective. Note that the sets $\SetRestrictionToPoint{H^\vb}{(\meinBildraum,0)}\times\R_\geq$ and $\left(\interior{\SetRestrictionToPoint{H^\vb}{(\meinBildraum,0)}}\cup\{0_\meinBildraum\}\right)\times\R_\geq$ are convex cones containing $\c\times\R_\geq$.

	\begin{Satz}\label{th:vectorialPenalisation:PEff:AddRemove:PenUnpen:PEff+/PEff}
		Under condition \eqref{ass:vectorialPenalisation:PEff:problemFormulation:VOPRestrVBVor} it is
		\begin{align}
			\begin{split}
				&\PEffHenig{\fvb}{\s}{\c\times\R_\geq}\setminus\PEffHenig{f}{\s}{\c} \\
				&=\Menge{\x\in\PEffHenig{\fvb}{\s}{\c\times\R_\geq}}{
						&\forall H\in\HenigDilatingCones{\c}\ \exists y\in\s: \\
						&f(y)<_{H}f(\x),\ \nu(y)>\nu(\x)
					}
			\end{split}
		\end{align}
	\end{Satz}
	\begin{proof}
        Let $\x\in\PEffHenig{\fvb}{\s}{\c\times\R_\geq}\setminus\PEffHenig{f}{\s}{\c}$. By definition there exists $H^\vb\in\HenigDilatingCones{\c\times\R_\geq}$ such that $\x\in\Eff{\fvb}{\s}{H^\vb}$. Set $H_0\coloneqq\SetRestrictionToPoint{H^\vb}{(\meinBildraum,0)}$.
        By \Cref{th:vectorialPenalisation:PEff:AddRemove:structuredDilatingCones:H^ocross_hProperties}~\ref{item:vectorialPenalisation:PEff:AddRemove:structuredDilatingCones:H^ocross_hProperties:MengeProjHx=HenigDC(C)}, we have $H_0\in\HenigDilatingCones{\c}$.

        Now let $H\in\HenigDilatingCones{\c}$ be arbitrary and set $K\coloneqq H\cap H_0$. By \Cref{th:PEff:HenigDilatingConesProperties}~\ref{item:PEff:HenigDilatingConesProperties:FiniteIntersection}, $K\in\HenigDilatingCones{\c}$. Since $\x\notin\PEffHenig{f}{\s}{\c}$, \Cref{th:PEff:PEffHenigEqui} implies $\x\notin\WEff{f}{\s}{K}$. Thus there exists $y\in\s$ such that
        \[
            f(\x)-f(y)\in\interior{K}.
        \]
        In particular, $f(y)<_H f(\x)$ because $K\subseteq H$. If $\nu(y)\leq\nu(\x)$, then \Cref{th:vectorialPenalisation:PEff:AddRemove:structuredDilatingCones:HxRInJ} gives $H_0\times\R_\geq\subseteq H^\vb$, and hence
        \[
            \fvb(\x)-\fvb(y)
            \in\bigl(H_0\setminus\{0_\meinBildraum\}\bigr)\times\R_\geq
            \subseteq H^\vb\setminus\{(0_\meinBildraum,0)\},
        \]
        contradicting $\x\in\Eff{\fvb}{\s}{H^\vb}$. Therefore $\nu(y)>\nu(\x)$. Since $H$ was arbitrary, this proves the inclusion.

        Conversely, let $\x$ belong to the set on the right-hand side. For every $H\in\HenigDilatingCones{\c}$ there exists $y\in\s$ with $f(y)<_H f(\x)$, so $\x\notin\WEff{f}{\s}{H}$ for every such $H$. By \Cref{th:PEff:PEffHenigEqui}, $\x\notin\PEffHenig{f}{\s}{\c}$. Together with the assumed penalised proper efficiency of $\x$, this proves the reverse inclusion.
    \end{proof}

	\begin{Bemerkung}
        \Cref{th:vectorialPenalisation:PEff:AddRemove:PenUnpen:PEff+/PEff} explains how proper efficiency may be created by penalisation. A point can become properly efficient for the penalised problem although it is not properly efficient for the original one if every improvement in the objective value, with respect to any Henig dilating cone, is accompanied by a larger value of the additional component. Hence, penalisation  can stabilise points by making the trade-off between objective improvement and the additional component explicit. This interpretation remains meaningful when the additional component represents another objective rather than a penalisation term.
    \end{Bemerkung}

	\begin{Bemerkung} \label{bem:vectorialPenalisation:PEff:AddRemove:PenUnpen:PEff+/PEffComparisonEffWEffSEff}
		Analogous characterisations of the additional solutions generated by penalisation are given in \cite{GuntherKobisSchmollingTammer2023VectorialpenalisationvectoroptimisationreallineartopologicalspacesjournalArticle} for weak efficiency (\cite[Proposition~4.16]{GuntherKobisSchmollingTammer2023VectorialpenalisationvectoroptimisationreallineartopologicalspacesjournalArticle}), strict efficiency (\cite[Proposition~4.17]{GuntherKobisSchmollingTammer2023VectorialpenalisationvectoroptimisationreallineartopologicalspacesjournalArticle}), and efficiency (\cite[Theorem~4.18]{GuntherKobisSchmollingTammer2023VectorialpenalisationvectoroptimisationreallineartopologicalspacesjournalArticle}). In each case, failure of the original efficiency notion is witnessed by a feasible point whose penalisation value is at least that of the candidate. This value must be strictly larger for strict efficiency and efficiency, whereas equality is allowed for weak efficiency. For Henig-proper efficiency, the witness condition must hold for every $H\in\HenigDilatingCones{\c}$, but the witness may depend on $H$.
	\end{Bemerkung}

	We now turn to the reverse relation and investigate when a point that is Henig-properly efficient for the original problem fails to remain Henig-properly efficient after adding the penalisation component.

	\begin{Satz}\label{th:vectorialPenalisation:PEff:AddRemove:PenUnpen:PEff/PEff+}
		Under condition \eqref{ass:vectorialPenalisation:PEff:problemFormulation:VOPRestrVBVor} it is
		\begin{align}
			\begin{split}
				&\PEffHenig{f}{\s}{\c}\setminus\PEffHenig{\fvb}{\s}{\c\times\R_\geq}\\
                &=\Menge{\x\in\PEffHenig{f}{\s}{\c}}{&\forall H^\vb\in\HenigDilatingCones{\c\times\R_\geq}\ \exists y\in\s:\\
				&\fvb(y)<_{H^\vb}\fvb(\x),\ \nu(y)<\nu(\x)}.
			\end{split} \label{gl:vectorialPenalisation:PEff:AddRemove:PenUnpen:PEff/PEff+}
		\end{align}
	\end{Satz}
    \begin{proof}
        Let $\x\in\PEffHenig{f}{\s}{\c}\setminus\PEffHenig{\fvb}{\s}{\c\times\R_\geq}$. By \Cref{th:PEff:PEffHenigEqui}, choose $H\in\HenigDilatingCones{\c}$ such that $\x\in\WEff{f}{\s}{H}$.
        Choose $h\in\interior[\meinBildraum]{H}$ and form $H^\vb_h$. By \Cref{th:vectorialPenalisation:PEff:AddRemove:structuredDilatingCones:H^ocross_hProperties}~\ref{item:vectorialPenalisation:PEff:AddRemove:structuredDilatingCones:H^ocross_hProperties:H^ocross_hINHeDC(CxR)}, $H^\vb_h\in\HenigDilatingCones{\c\times\R_\geq}$.

        Let $G\in\HenigDilatingCones{\c\times\R_\geq}$ be arbitrary and set $K\coloneqq G\cap H^\vb_h$. By \Cref{th:PEff:HenigDilatingConesProperties}~\ref{item:PEff:HenigDilatingConesProperties:FiniteIntersection}, applied to the ordering cone $\c\times\R_\geq$, $K\in\HenigDilatingCones{\c\times\R_\geq}$. Since $\x\notin\PEffHenig{\fvb}{\s}{\c\times\R_\geq}$, \Cref{th:PEff:PEffHenigEqui} yields $\x\notin\WEff{\fvb}{\s}{K}$, so there exists $y\in\s$ with $\fvb(y)<_K\fvb(\x)$. Consequently,
        \[
            \fvb(\x)-\fvb(y)
            =(f(\x)-f(y),\nu(\x)-\nu(y))
            \in\interior{K}\subseteq\interior{G}\cap\interior{H^\vb_h}.
        \]
        If $\nu(y)\geq\nu(\x)$, then \Cref{th:vectorialPenalisation:PEff:AddRemove:structuredDilatingCones:H^ocross_hProperties}~\eqref{gl:vectorialPenalisation:PEff:AddRemove:structuredDilatingCones:H^ocross_hProperties:InteriorNegativeHalfspace} implies $f(\x)-f(y)\in\interior[\meinBildraum]{H}$, contradicting $\x\in\WEff{f}{\s}{H}$. Hence $\nu(y)<\nu(\x)$. Since the product difference also belongs to $\interior{G}$ and $G$ was arbitrary, this proves the inclusion from left to right.

        Conversely, let $\x$ belong to the right-hand side. For every $H^\vb\in\HenigDilatingCones{\c\times\R_\geq}$, the condition provides $y\in\s$ with $\fvb(y)<_{H^\vb}\fvb(\x)$. This contradicts efficiency with respect to $H^\vb$, so $\x\notin\PEffHenig{\fvb}{\s}{\c\times\R_\geq}$. Since $\x\in\PEffHenig{f}{\s}{\c}$, it belongs to the left-hand side.
    \end{proof}

	\begin{Bemerkung}
        \Cref{th:vectorialPenalisation:PEff:AddRemove:PenUnpen:PEff/PEff+} characterises loss of proper efficiency by requiring, for every product-space Henig dilating cone, a strict improving point whose penalty value is lower. In particular, such a loss entails the existence of a feasible point with lower penalty.  For both results, the additional component can also be an additional objective and need not be a penalty. Moreover, the characterisations in \Cref{th:vectorialPenalisation:PEff:AddRemove:PenUnpen:PEff+/PEff} and \Cref{th:vectorialPenalisation:PEff:AddRemove:PenUnpen:PEff/PEff+} remain valid if the displayed relations $<_{H}$ and $<_{H^\vb}$ are replaced by $\lneq_{H}$ and $\lneq_{H^\vb}$, respectively. Indeed, by \Cref{th:PEff:PEffHenigEqui}, Henig-proper efficiency can be characterised using either weak efficiency or efficiency with respect to Henig dilating cones. Using weak efficiency yields the interior-cone improvements $<$ in the stated versions, whereas using efficiency yields the nonzero-cone improvements $\lneq$.
    \end{Bemerkung}

	\begin{Bemerkung}
        Even though based on \Cref{th:PEff:PEffHenigEqui} it is enough to consider only weakly efficient solutions to some Henig dilating cone, we do not get the property $\PEffHenig{f}{\s}{\c}\setminus\PEffHenig{\fvb}{\s}{\c\times\R_\geq}=\emptyset$ as we get it for the case of weak efficiency (see \cite[Prop.~4.16]{GuntherKobisSchmollingTammer2023VectorialpenalisationvectoroptimisationreallineartopologicalspacesjournalArticle}).
        The reason is that the weak-efficiency inclusion concerns the fixed ordering cones $\c\times\R_\geq$ and $\c$, whereas Henig-proper efficiency permits a different dilating cone in the product space. If $H\in\HenigDilatingCones{\c}$ certifies the proper efficiency of $\x$ and $H^\vb\in\HenigDilatingCones{\c\times\R_\geq}$ is chosen small enough to be contained in some $H^\vb_h\in\HenigDilatingConesVB{\c}$, then an improvement $\fvb(\x)-\fvb(y)\in H^\vb\setminus\{0_\meinBildraum\}$ contradicts the $H$-efficiency of $\fvb(\x)$ only when $\nu(y)\geq\nu(\x)$. When $\nu(y)<\nu(\x)$, the positive last component of $\fvb(\x)-\fvb(y)$ may prevent this difference from belonging to $H\times\R_>$ (see \Cref{th:vectorialPenalisation:PEff:AddRemove:structuredDilatingCones:HxRInJ}), so the objective improvement need not destroy proper efficiency for the original problem. This is precisely the mechanism described in \Cref{th:vectorialPenalisation:PEff:AddRemove:PenUnpen:PEff/PEff+}.

		For ordinary efficiency, \textcite[Th.~4.18]{GuntherKobisSchmollingTammer2023VectorialpenalisationvectoroptimisationreallineartopologicalspacesjournalArticle} proved the analogue of \Cref{th:vectorialPenalisation:PEff:AddRemove:PenUnpen:PEff/PEff+}  under the additional requirement $f(y)=f(\x)$ on the improving point $y$.
    \end{Bemerkung}

	The following example shows that, for proper efficiency, neither this equality nor the inclusion from weak efficiency need to hold without further assumptions. More precisely, even for a closed cone, a closed feasible set, and a bounded continuous objective function, we cannot guarantee the existence of a point $y\in\s$ such that $f(y)=f(\x)$ and $\nu(y)<\nu(\x)$.

	\begin{Beispiel}\label{bsp:vectorialPenalisation:PEff:AddRemove:PenUnpen:NoEqualObjectiveValue}
		In $\R^2$ we consider the natural ordering cone $\c=\R^2_\geq$ and the function
        \[f:\R^2\to\R^2,\qquad f(x)\coloneqq\begin{cases}
            \left(x_1,\sqrt{|x_1x_2|}\right)&\text{if}\ \norm{x}_{\ell^\infty}\leq1,\\
            \frac{\left(x_1,\sqrt{|x_1x_2|}\right)}{\norm{x}_{\ell^\infty}}&\text{if}\ \norm{x}_{\ell^\infty}>1.
        \end{cases}\]
        Let $\s\coloneqq\Menge{x\in\R^2}{x_1\geq0,\ x_2\geq x_1}\subset\c$, then for $\x\coloneqq0_{\R^2}\in\s$ it is $f(\x)=0_{\R^2}$ and we get $\x\in\PEffHenig{f}{\s}{\c}$. As a Henig dilating cone we can choose the half space $H=\Menge{v\in\R^2}{\skp{(1,1)}{v}\geq0}$, obviously satisfying $H\in\HenigDilatingCones{\c}$. For any $y\in\s\ohne\{0_{\R^2}\}$ it is $\skp{(1,1)}{f(y)}\geq0$ and $\skp{(1,1)}{f(y)}=0$ if and only if $y_1=0$, so that $f(y)=0_{\R^2}$. Thus, $\x\in\Eff{f}{\s}{H}$, proving $\x\in\PEffHenig{f}{\s}{\c}$.

        We now take $\nu:\R^2\to\R$, $\nu(x)\coloneqq-\norm{x}_{\ell^\infty}$. Suppose that there exists $H^\vb\in\HenigDilatingCones{\c\times\R_\geq}$ with $\x\in\Eff{\fvb}{\s}{H^\vb}$. Since
        \[
            (0,0,1)\in(\c\times\R_\geq)\setminus\{0_{\R^3}\}\subseteq\interior{H^\vb},
        \]
        there exists $n\in\N$ such that $\left(-\tfrac{1}{n},-\tfrac{1}{n},1\right)\in H^\vb$ and then by scaling $(-1,-1,n)\in H^\vb$. For $y=(n,n)\in\s$, we obtain
        \[
            \fvb(y)=(1,1,-n)\in\fvb(\x)-H^\vb\setminus\{0_{\R^3}\}=-H^\vb\setminus\{0_{\R^3}\},
        \]
        contradicting $\x\in\Eff{\fvb}{\s}{H^\vb}$. Hence $\x\notin\PEffHenig{\fvb}{\s}{\c\times\R_\geq}$. Furthermore, it is $f(y)\neq0_{\R^2}$ for every $y\in\s\setminus\{0_{\R^2}\}$.

	\end{Beispiel}

    We can also find an example of a bounded continuous objective function, a compact feasible set and a closed cone for which all the counter examples $y$ for the proper efficiency in \eqref{P:vectorialPenalisation:PEff:AddRemove:PvbS} of some properly efficient point $\x$ of \eqref{P:P} have different function values $f(y)\neq f(\x)$.
	
    \begin{Beispiel}\label{bsp:vectorialPenalisation:PEff:AddRemove:PenUnpen:DifferentObjectiveValues}
        Consider the identity function $f:\R\to\R$, $f(x)\coloneqq x$, over the feasible set $\s=[0,1]$ with the ordering cone $\c=\R_\geq$. We can see that the unique minimum of $f$ over $\s$ is the point $\x=0$. Since $\c$ is its own dilating cone we thus get $\x\in\PEffHenig{f}{\s}{\c}$. Now we introduce $\nu:\R\to\R$, $\nu(x)\coloneqq-\sqrt{|x|}$. Suppose there were $H\in\HenigDilatingCones{\c\times\R_\geq}$ with $\x\in\Eff{\fvb}{\s}{H}$. Then it would be $(0,1)\in\interior H$ and thus there would exists $\epsilon\in(0,1)$ such that $(-\sqrt\epsilon,1)\in H$ and by scaling also $(-\epsilon,\sqrt\epsilon)\in H$. It is however $\fvb(\epsilon)=(\epsilon,-\sqrt\epsilon)\in\fvb(\x)-H\setminus\{0_{\R^2}\}$. Thus $\x\notin\PEffHenig{\fvb}{\s}{\c\times\R_\geq}$, while $f(y)>f(\x)$ for every $y\in\s\setminus\{0\}$.
    \end{Beispiel}

Even though the corresponding equality need not hold for Henig-proper efficiency, the following lemma establishes at least one inclusion in the analogous characterisation.

    \begin{Lemma}\label{th:vectorialPenalisation:PEff:AddRemove:PenUnpen:PEff/PEff+Equal}
        Under condition \eqref{ass:vectorialPenalisation:PEff:problemFormulation:VOPRestrVBVor} it is
		\begin{equation}
			\begin{split}
				&\Menge{\x\in\PEffHenig{f}{\s}{\c}}{\exists y\in\s:\ f(y)=f(\x),\ \nu(y)<\nu(\x)} \\ 
				&\subseteq \PEffHenig{f}{\s}{\c}\setminus\PEffHenig{\fvb}{\s}{\c\times\R_\geq}.
			\end{split}
            \label{gl:vectorialPenalisation:PEff:AddRemove:PenUnpen:PEff/PEff+EqualSamef}
		\end{equation}
        Consequently, if
        \[
			\begin{split}
				&\Menge{\x\in\PEffHenig{f}{\s}{\c}}{\exists y\in\s:\ \nu(y)<\nu(\x)} \\
            &\subseteq \Menge{\x\in\PEffHenig{f}{\s}{\c}}{\exists y\in\s:\ f(y)=f(\x),\ \nu(y)<\nu(\x)}.
			\end{split}
        \]
		then it is
        \begin{align*}
            \begin{split}
                &\PEffHenig{f}{\s}{\c}\setminus\PEffHenig{\fvb}{\s}{\c\times\R_\geq}\\
                &=\Menge{\x\in\PEffHenig{f}{\s}{\c}}{\exists y\in\s:\ f(y)=f(\x),\ \nu(y)<\nu(\x)}
            \end{split}
        \end{align*}
        Equivalently, the equality holds if and only if
		\begin{align*}
			&\PEffHenig{f}{\s}{\c}\setminus\PEffHenig{\fvb}{\s}{\c\times\R_\geq} \\
            &\subseteq
            \Menge{\x\in\PEffHenig{f}{\s}{\c}}{\exists y\in\s:\ f(y)=f(\x),\ \nu(y)<\nu(\x)}.
		\end{align*}
    \end{Lemma}
    \begin{proof}
        Let $H^\vb\in\HenigDilatingCones{\c\times\R_\geq}$ be arbitrary. Since $f(y)=f(\x)$ and $\nu(y)<\nu(\x)$, we obtain
        \[
            \fvb(\x)-\fvb(y)=(0_{\meinBildraum},\nu(\x)-\nu(y))\in\c\times\R_\geq\subseteq\interior[\meinBildraum\times\R]{H^\vb}.
        \]
        This proves $\x\notin\WEff{\fvb}{\s}{H^\vb}$. Since $H^\vb$ was arbitrary, \Cref{th:PEff:PEffHenigEqui} implies $\x\notin\PEffHenig{\fvb}{\s}{\c\times\R_\geq}$
        and thus \eqref{gl:vectorialPenalisation:PEff:AddRemove:PenUnpen:PEff/PEff+EqualSamef} follows.
		
		The subsequent equality statement is immediate from \Cref{th:vectorialPenalisation:PEff:AddRemove:PenUnpen:PEff/PEff+} and \Cref{th:vectorialPenalisation:PEff:AddRemove:structuredDilatingCones:NonemptyDilatingConesProductEquivalence}. The final equivalence is just the corresponding reformulation of equality in terms of the missing inclusion.
    \end{proof}

	Since the equality condition is too strong, a first generalisation is to allow the entire boundary of a dilating cone.

	\begin{Korollar}\label{th:vectorialPenalisation:PEff:AddRemove:PenUnpen:PEff/PEff+ConewiseBoundaryCondition}
        Under condition \eqref{ass:vectorialPenalisation:PEff:problemFormulation:VOPRestrVBVor}, it holds that
        \begin{align}
            \begin{split}
                &\Menge{\x\in\PEffHenig{f}{\s}{\c}}{
                    &\forall H\in\HenigDilatingCones{\c}:\ \x\in\WEff{f}{\s}{H}\Rightarrow\\
                    &\exists y\in\s:\ \nu(y)<\nu(\x),\ f(\x)-f(y)\in\boundary[\meinBildraum]{H}
                }\\
                &\subseteq\PEffHenig{f}{\s}{\c}\setminus\PEffHenig{\fvb}{\s}{\c\times\R_\geq}.
            \end{split}
            \label{gl:vectorialPenalisation:PEff:AddRemove:PenUnpen:PEff/PEff+ConewiseBoundaryCondition:Inclusion}
        \end{align}
    \end{Korollar}
    \begin{proof}
        Let $\x$ belong to the set on the left-hand side. Let $H^\vb\in\HenigDilatingCones{\c\times\R_\geq}$ be arbitrary and set $H_0\coloneqq\SetRestrictionToPoint{H^\vb}{(\meinBildraum,0)}$.
        By \Cref{th:vectorialPenalisation:PEff:AddRemove:structuredDilatingCones:H^ocross_hProperties}~\ref{item:vectorialPenalisation:PEff:AddRemove:structuredDilatingCones:H^ocross_hProperties:MengeProjHx=HenigDC(C)}, $H_0\in\HenigDilatingCones{\c}$. Since $\x\in\PEffHenig{f}{\s}{\c}$, there exists $H_1\in\HenigDilatingCones{\c}$ such that $\x\in\WEff{f}{\s}{H_1}$. Set $K\coloneqq H_0\cap H_1$. By \Cref{th:PEff:HenigDilatingConesProperties}~\ref{item:PEff:HenigDilatingConesProperties:FiniteIntersection}, $K\in\HenigDilatingCones{\c}$. Moreover, \Cref{th:BeziehungEff}~\ref{item:BeziehungEff:Eff(C2)inEff(C1)} and $K\subseteq H_1$ imply $\x\in\WEff{f}{\s}{K}$. By the stated condition applied to $K$, there exists $y\in\s$ such that $\nu(y)<\nu(\x)$ and
        \[
            f(\x)-f(y)\in\boundary{K}\subseteq\cl{K}\subseteq\cl{H_0}.
        \]
        Since $H_0$ is the zero-level section of $H^\vb$, it is $\left(f(\x)-f(y),0\right)\in\cl{H^\vb}$. Moreover,
        \[
            (0_{\meinBildraum},\nu(\x)-\nu(y))
            \in(\c\times\R_\geq)\setminus\{(0_{\meinBildraum},0)\}
            \subseteq\interior{H^\vb}.
        \]
        As by \Cref{th:Mengeneigenschaften:Kegel:konvexÄquivalenz} applied on $\cl{H^\vb}$ it is $\cl{H^\vb}+\interior{H^\vb}\subseteq\interior{H^\vb}$, it follows that
        \[
            \fvb(\x)-\fvb(y)
            =\left(f(\x)-f(y),0\right)+(0_{\meinBildraum},\nu(\x)-\nu(y))
            \in\interior{H^\vb}.
        \]
        Since $H^\vb$ was arbitrary, \Cref{th:PEff:PEffHenigEqui} applied to the penalised objective yields $\x\notin\PEffHenig{\fvb}{\s}{\c\times\R_\geq}$, which proves the inclusion.
    \end{proof}

	\subsection{Vectorial penalisation - relationships between the solution sets}\label{sec:vectorialPenalisation:PEff:solutionSetRelations}

    We now turn to the main result of the paper, namely a characterisation of the set of Henig-properly efficient points of the constrained problem in terms of the solution sets of at most two associated unconstrained problems. The following theorem gives the inclusion relations obtained through vectorial penalisation and a complete representation under an additional local-global type condition ensuring that every Henig-properly efficient solution of the constrained problem also belongs to the Henig-properly efficient solution set of the penalised problem on $\calD$.

	\begin{Satz}\label{th:vectorialPenalisation:PEff:solutionSetRelations:BeziehungenPEffs}
		Assume \eqref{ass:vectorialPenalisation:PEff:problemFormulation:VOPRestrVBVor}.
		\begin{enumerate}
			\item \label{item:vectorialPenalisation:PEff:solutionSetRelations:BeziehungenPEffs:MaximumTransfer}
			It is
			\begin{equation}
				\left(\argmax\limits_{x\in\s}\nu(x)\right)
				\cap\s\cap\PEffHenig{\fvb}{\calD}{\c\times\R_\geq}
				\subseteq\PEffHenig{f}{\s}{\c}.
				\label{gl:vectorialPenalisation:PEff:solutionSetRelations:BeziehungenPEffs:MaximumTransfer}
			\end{equation}

			\item \label{item:vectorialPenalisation:PEff:solutionSetRelations:BeziehungenPEffs:A1}
			Let $\s$ be closed and assume \eqref{ass:vectorialPenalisation:PEff:penalisationFunctions:A1}. Then
            \begin{adjustwidth}{-\leftmargini}{0pt}
            \setlength{\abovedisplayskip}{\smallskipamount}
            \setlength{\abovedisplayshortskip}{\smallskipamount}
            \begin{align}
                \PEffHenig{f}{\s}{\c}&\supseteq\left(\interior[\meinUrbildraum]{\s}\cap\PEffHenig{f}{\calD}{\c}\right)\cup\left(\boundary[\meinUrbildraum]{\s}\cap\PEffHenig{\fvb}{\calD}{\c\times\R_\geq}\right) \label{gl:vectorialPenalisation:PEff:solutionSetRelations:BeziehungenPEffs:A1:DIsub} \\
                \PEffHenig{f}{\s}{\c}&\supseteq\left(\s\cap\PEffHenig{f}{\calD}{\c}\right)\cup\left(\boundary[\meinUrbildraum]{\s}\cap\PEffHenig{\fvb}{\calD}{\c\times\R_\geq}\right) \label{gl:vectorialPenalisation:PEff:solutionSetRelations:BeziehungenPEffs:A1:Dsub}
            \end{align}
            \end{adjustwidth}


			\item \label{item:vectorialPenalisation:PEff:solutionSetRelations:BeziehungenPEffs:A3}
			Assume $\nu\in \mathzapf{A}_3(\s,\calD)$. Then we have
			\begin{align}
				\s\cap\PEffHenig{\fvb}{\calD}{\c\times\R_\geq}
				&\subseteq\PEffHenig{f}{\s}{\c},
				\label{gl:vectorialPenalisation:PEff:solutionSetRelations:BeziehungenPEffs:A3:ScapPEffInPEffS+beiA3} \\
                \s\cap\PEffHenig{f}{\calD}{\c}
				&\subseteq\s\cap\PEffHenig{\fvb}{\calD}{\c\times\R_\geq}.
                \label{gl:vectorialPenalisation:PEff:solutionSetRelations:BeziehungenPEffs:A3:ScapPEffDInScapPEffvbD}
            \end{align}
            If \eqref{gl:MOP:Eff:PEff:SPEffD=PEffS} holds, 
            then
            \begin{equation}
				\PEffHenig{f}{\s}{\c}=\s\cap\PEffHenig{\fvb}{\calD}{\c\times\R_\geq}. \label{gl:vectorialPenalisation:PEff:solutionSetRelations:BeziehungenPEffs:A3:PEffS=ScapPEff+beiA3}
			\end{equation}
			If $\s$ is closed, then we have \eqref{gl:vectorialPenalisation:PEff:solutionSetRelations:BeziehungenPEffs:A1:DIsub} and \eqref{gl:vectorialPenalisation:PEff:solutionSetRelations:BeziehungenPEffs:A1:Dsub}.
			If additionally \eqref{gl:vectorialPenalisation:PEff:solutionSetRelations:BeziehungenPEffs:A3:PEffS=ScapPEff+beiA3} holds and also \eqref{gl:MOP:Eff:PEff:PEff/PEffInBD} holds, so
			\begin{equation}
				\PEffHenig{f}{\s}{\c}\setminus\PEffHenig{f}{\calD}{\c}\subseteq\boundary[\meinUrbildraum]{\s} \label{gl:vectorialPenalisation:PEff:solutionSetRelations:BeziehungenPEffs:A3:PEffS/PEffDForm}
			\end{equation}
            is valid, then we even get the equalities
            \begin{adjustwidth}{-\leftmargini}{0pt} ~\\[-2\baselineskip]
            \begin{align}
				\PEffHenig{f}{\s}{\c}&=\left(\interior[\meinUrbildraum]{\s}\cap\PEffHenig{f}{\calD}{\c}\right)\cup\left(\boundary[\meinUrbildraum]{\s}\cap\PEffHenig{\fvb}{\calD}{\c\times\R_\geq}\right) \label{gl:vectorialPenalisation:PEff:solutionSetRelations:BeziehungenPEffs:A3:DI=}\\
				\PEffHenig{f}{\s}{\c}&=\left(\s\cap\PEffHenig{f}{\calD}{\c}\right)\cup\left(\boundary[\meinUrbildraum]{\s}\cap\PEffHenig{\fvb}{\calD}{\c\times\R_\geq}\right) \label{gl:vectorialPenalisation:PEff:solutionSetRelations:BeziehungenPEffs:A3:D=}
            \end{align}
            \end{adjustwidth}
		\end{enumerate}
	\end{Satz}
	\begin{proof}
		\begin{enumerate}[leftmargin=*]
			\item[\ref{item:vectorialPenalisation:PEff:solutionSetRelations:BeziehungenPEffs:MaximumTransfer}]
			Let
			\[
				\x\in\left(\argmax\limits_{x\in\s}\nu(x)\right)
				\cap\s\cap\PEffHenig{\fvb}{\calD}{\c\times\R_\geq}.
			\]
			By \Cref{th:MOP:Eff:PEff:ScapPEffInPEff}, $\x\in\PEffHenig{\fvb}{\s}{\c\times\R_\geq}$. Moreover, $\x\in\PEffHenig{\fvb}{\calD}{\c\times\R_\geq}$ implies $\HenigDilatingCones{\c\times\R_\geq}\neq\emptyset$, so \Cref{th:vectorialPenalisation:PEff:AddRemove:structuredDilatingCones:NonemptyDilatingConesProductEquivalence} gives $\HenigDilatingCones{\c}\neq\emptyset$. If $\x\notin\PEffHenig{f}{\s}{\c}$, then \Cref{th:vectorialPenalisation:PEff:AddRemove:PenUnpen:PEff+/PEff} yields, for every $H\in\HenigDilatingCones{\c}$, a point $y\in\s$ satisfying $\nu(y)>\nu(\x)$. This contradicts the maximality of $\nu(\x)$ on $\s$ and proves \eqref{gl:vectorialPenalisation:PEff:solutionSetRelations:BeziehungenPEffs:MaximumTransfer}.

			\item[\ref{item:vectorialPenalisation:PEff:solutionSetRelations:BeziehungenPEffs:A1}]
			By \Cref{th:MOP:Eff:PEff:ScapPEffInPEff} we get $\s\cap\PEffHenig{f}{\calD}{\c}\subseteq\PEffHenig{f}{\s}{\c}$. For any $\x\in\boundary[\meinUrbildraum]{\s}\cap\PEffHenig{\fvb}{\calD}{\c\times\R_\geq}$ we have $\x\in\s$ and since for any other $y\in\s$ by \eqref{ass:vectorialPenalisation:PEff:penalisationFunctions:A1} it is $\nu(y)\leq\nu(\x)$ with \Cref{th:vectorialPenalisation:PEff:AddRemove:PenUnpen:PEff+/PEff} and \Cref{th:MOP:Eff:PEff:ScapPEffInPEff} this implies $\x\in\PEffHenig{f}{\s}{\c}$, proving \eqref{gl:vectorialPenalisation:PEff:solutionSetRelations:BeziehungenPEffs:A1:Dsub} and thus also \eqref{gl:vectorialPenalisation:PEff:solutionSetRelations:BeziehungenPEffs:A1:DIsub}.

			\item[\ref{item:vectorialPenalisation:PEff:solutionSetRelations:BeziehungenPEffs:A3}]
			Under \eqref{ass:vectorialPenalisation:PEff:penalisationFunctions:A3}, \Cref{th:MOP:Eff:PEff:ScapPEffInPEff} and \Cref{th:vectorialPenalisation:PEff:AddRemove:PenUnpen:PEff+/PEff} yield
			\[\s\cap\PEffHenig{\fvb}{\calD}{\c\times\R_\geq}\subseteq\PEffHenig{\fvb}{\s}{\c\times\R_\geq}\subseteq\PEffHenig{f}{\s}{\c}.\]
            This proves \eqref{gl:vectorialPenalisation:PEff:solutionSetRelations:BeziehungenPEffs:A3:ScapPEffInPEffS+beiA3}.

            \Cref{th:vectorialPenalisation:PEff:AddRemove:PenUnpen:PEff/PEff+} immediately gives \eqref{gl:vectorialPenalisation:PEff:solutionSetRelations:BeziehungenPEffs:A3:ScapPEffDInScapPEffvbD}.


            If additionally \eqref{gl:MOP:Eff:PEff:SPEffD=PEffS} is valid, then \eqref{gl:vectorialPenalisation:PEff:solutionSetRelations:BeziehungenPEffs:A3:ScapPEffDInScapPEffvbD} gives
            \[
                \PEffHenig{f}{\s}{\c}
                =\s\cap\PEffHenig{f}{\calD}{\c}
                \subseteq\s\cap\PEffHenig{\fvb}{\calD}{\c\times\R_\geq}.
            \]
            Together with \eqref{gl:vectorialPenalisation:PEff:solutionSetRelations:BeziehungenPEffs:A3:ScapPEffInPEffS+beiA3}, this proves \eqref{gl:vectorialPenalisation:PEff:solutionSetRelations:BeziehungenPEffs:A3:PEffS=ScapPEff+beiA3}.
			
			If $\s$ is closed, \eqref{ass:vectorialPenalisation:PEff:penalisationFunctions:A3} implies \eqref{ass:vectorialPenalisation:PEff:penalisationFunctions:A1} and hence by part~\ref{item:vectorialPenalisation:PEff:solutionSetRelations:BeziehungenPEffs:A1} we get both \eqref{gl:vectorialPenalisation:PEff:solutionSetRelations:BeziehungenPEffs:A1:DIsub} and \eqref{gl:vectorialPenalisation:PEff:solutionSetRelations:BeziehungenPEffs:A1:Dsub}. Furthermore, assume \eqref{gl:vectorialPenalisation:PEff:solutionSetRelations:BeziehungenPEffs:A3:PEffS=ScapPEff+beiA3} and \eqref{gl:vectorialPenalisation:PEff:solutionSetRelations:BeziehungenPEffs:A3:PEffS/PEffDForm} and let $\x\in\PEffHenig{f}{\s}{\c}$. Since $\s$ is closed, either $\x\in\interior[\meinUrbildraum]{\s}$ or $\x\in\boundary[\meinUrbildraum]{\s}$. In the first case, \eqref{gl:vectorialPenalisation:PEff:solutionSetRelations:BeziehungenPEffs:A3:PEffS/PEffDForm} implies $\x\in\PEffHenig{f}{\calD}{\c}$. In the second case, \eqref{gl:vectorialPenalisation:PEff:solutionSetRelations:BeziehungenPEffs:A3:PEffS=ScapPEff+beiA3} implies $\x\in\PEffHenig{\fvb}{\calD}{\c\times\R_\geq}$. This proves the inclusion in \eqref{gl:vectorialPenalisation:PEff:solutionSetRelations:BeziehungenPEffs:A3:DI=} and thus also in \eqref{gl:vectorialPenalisation:PEff:solutionSetRelations:BeziehungenPEffs:A3:D=}. \qedhere
		\end{enumerate}
	\end{proof}
	
    The following example shows that the inclusions in part~\ref{item:vectorialPenalisation:PEff:solutionSetRelations:BeziehungenPEffs:A1} of \Cref{th:vectorialPenalisation:PEff:solutionSetRelations:BeziehungenPEffs} can be strict. Even with $\nu\in\mathzapf{A}_1(\s,\calD)$, and the boundary condition \eqref{gl:MOP:Eff:PEff:PEff/PEffInBD}, neither exact decomposition \eqref{gl:vectorialPenalisation:PEff:solutionSetRelations:BeziehungenPEffs:A3:DI=} nor \eqref{gl:vectorialPenalisation:PEff:solutionSetRelations:BeziehungenPEffs:A3:D=} need hold. Also, coincidence of the two decomposition right-hand sides does not by itself ensure that their common value equals the constrained properly efficient set.

    \begin{Beispiel}\label{bsp:vectorialPenalisation:PEff:solutionSetRelations:A1BoundedNormalizedDifferencesNotExact}
        Let $\meinUrbildraum=\calD=\R$, $\meinBildraum=\R^2$, $\s=[0,\infty)$, and $\c=\R^2_{\geq}$. Define $\nu(x)=-x$ and
        \[
            f(x)=
            \begin{cases}
                (x,x), & x<0,\\
                \left(\dfrac{x^2}{1+x^2},0\right), & x\geq0.
            \end{cases}
        \]
        It is $\nu\in\mathzapf A_1(\s,\calD)$, since $\lev_\leq(\calD,\nu,\nu(0))=\s$.
        $\PEffHenig{f}{\s}{\c}=\{0\}$, because the point $0$ is efficient, for example with respect to the Henig dilating cone $H=\Menge{(u,v)\in\R^2}{u+v\geq0}$, while every $x>0$ is dominated by $0$ for the original objective. Moreover, $\s\cap\PEffHenig{f}{\calD}{\c}=\emptyset$, since every $y<0$ strictly dominates $0$ on $\calD$ with $f(0)-f(y)=(-y,-y)\in\c\setminus\{0\}$. Also, $0\notin\PEffHenig{\fvb}{\calD}{\c\times\R_\geq}$. Indeed, for every $G\in\HenigDilatingCones{\c\times\R_\geq}$, $(0,0,1)\in\interior{G}$, and
        \[
            \fvb(0)-\fvb(y)=\left(-\frac{y^2}{1+y^2},0,y\right)
            =y\left(-\frac{y}{1+y^2},0,1\right)
            \in\interior{G}
        \]
        for all sufficiently small $y>0$, by openness of $\interior{G}$. Thus the right-hand sides of both decomposition formulas \eqref{gl:vectorialPenalisation:PEff:solutionSetRelations:BeziehungenPEffs:A3:DI=} and \eqref{gl:vectorialPenalisation:PEff:solutionSetRelations:BeziehungenPEffs:A3:D=} are empty, while their left-hand side is $\{0\}$. Hence even for a continuous objective function the boundary condition in \eqref{gl:MOP:Eff:PEff:PEff/PEffInBD} does not suffice for either equality.
    \end{Beispiel}

	The additional exactness condition \eqref{gl:MOP:Eff:PEff:SPEffD=PEffS} in \Cref{th:vectorialPenalisation:PEff:solutionSetRelations:BeziehungenPEffs}~\ref{item:vectorialPenalisation:PEff:solutionSetRelations:BeziehungenPEffs:A3} to get the exact representation \eqref{gl:vectorialPenalisation:PEff:solutionSetRelations:BeziehungenPEffs:A3:PEffS=ScapPEff+beiA3} cannot be omitted, even for continuous convex data.

	\begin{Beispiel}\label{bsp:vectorialPenalisation:PEff:solutionSetRelations:A3DoesNotImplyProperExactness}
		Let $\meinUrbildraum=\meinBildraum=\R$, $\calD=[0,1]$, $\s=\{0\}$, $\c=\R_\geq$, and define
		\[
			f(t)=-\sqrt{t},
			\qquad
			\nu(t)=t.
		\]
		Then $f$ is continuous and $\c$-convex on $\calD$, $\s$ is closed, and $\nu\in\mathzapf{A}_3(\s,\calD)$. Since $\s$ is a singleton, $0\in\PEffHenig{f}{\s}{\c}$. We show, however, that
		\[
			0\notin\PEffHenig{\fvb}{\calD}{\R^2_\geq}.
		\]
		Indeed, let $K\in\HenigDilatingCones{\R^2_\geq}$. Then $(1,0)\in\interior[\R^2]{K}$. Hence $(1,-\sqrt{t})\in\interior[\R^2]{K}$ for all sufficiently small $t>0$. Since the interior of a convex cone is positively homogeneous,
		\[
			\fvb(0)-\fvb(t)=(\sqrt{t},-t)=\sqrt{t}(1,-\sqrt{t})\in\interior[\R^2]{K}.
		\]
		Thus $0\notin\WEff{\fvb}{\calD}{K}$. Since $K$ was arbitrary, \Cref{th:PEff:PEffHenigEqui} gives the claim.
	\end{Beispiel}

    The preceding convexity-based condition \eqref{gl:MOP:Eff:PEff:SPEffD=PEffS} is one way to obtain the equality in \eqref{gl:vectorialPenalisation:PEff:solutionSetRelations:BeziehungenPEffs:A3:PEffS=ScapPEff+beiA3}. The following result provides an alternative pointwise criterion based on a directional separation. For each residual constrained Henig-properly efficient point, a single direction $h$ in the interior of a suitable dilating cone $H$ is used to displace the competing values. The corrected value $f(y)-(\nu(\x)-\nu(y))h$ remains outside the strict lower set $f(\x)-\interior[\meinBildraum]{H}$. In this sense, $h$ acts as a fixed dilation direction that pushes the competing values away from $f(\x)-H$.

    \begin{Korollar}\label{th:vectorialPenalisation:PEff:solutionSetRelations:A3Exactness}
        Assume \eqref{ass:vectorialPenalisation:PEff:problemFormulation:VOPRestrVBVor} and $\nu\in \mathzapf{A}_3(\s,\calD)$. If, in addition, for every $\x\in\PEffHenig{f}{\s}{\c}\setminus\PEffHenig{f}{\calD}{\c}$ there exist $H\in\HenigDilatingCones{\c}$ and $h\in\interior[\meinBildraum]{H}$ such that $\x\in\WEff{f}{\s}{H}$ and
        \begin{equation}
            f(\x)-f(y)\notin\interior[\meinBildraum]{H}-(\nu(\x)-\nu(y))h
            \quad\text{for all }y\in\calD,
            \label{gl:vectorialPenalisation:PEff:solutionSetRelations:A3Exactness}
        \end{equation}
        then \eqref{gl:vectorialPenalisation:PEff:solutionSetRelations:BeziehungenPEffs:A3:PEffS=ScapPEff+beiA3} holds.
    \end{Korollar}
    \begin{proof}
        The reverse inclusion in \eqref{gl:vectorialPenalisation:PEff:solutionSetRelations:BeziehungenPEffs:A3:PEffS=ScapPEff+beiA3} follows from \eqref{gl:vectorialPenalisation:PEff:solutionSetRelations:BeziehungenPEffs:A3:ScapPEffInPEffS+beiA3}.
        Now assume \eqref{gl:vectorialPenalisation:PEff:solutionSetRelations:A3Exactness} and let $\x\in\PEffHenig{f}{\s}{\c}$. If $\x\in\PEffHenig{f}{\calD}{\c}$, then $\x\in\PEffHenig{\fvb}{\calD}{\c\times\R_\geq}$ by \eqref{gl:vectorialPenalisation:PEff:solutionSetRelations:BeziehungenPEffs:A3:ScapPEffDInScapPEffvbD}. Otherwise, choose $H$ and $h$ as in the condition. By the definition of $H_h^\vb$ and \eqref{gl:vectorialPenalisation:PEff:AddRemove:structuredDilatingCones:HeDC(C)HeDC(CxR):Int(Hhh)}, condition \eqref{gl:vectorialPenalisation:PEff:solutionSetRelations:A3Exactness} is equivalent to
			\[
				\fvb(\x)-\fvb(y)\notin\interior[\meinBildraum\times\R]{H_h^\vb}
				\quad\text{for all }y\in\calD.
			\]
			Hence $\x\in\WEff{\fvb}{\calD}{H_h^\vb}$. Moreover, \Cref{th:vectorialPenalisation:PEff:AddRemove:structuredDilatingCones:H^ocross_hProperties}~\ref{item:vectorialPenalisation:PEff:AddRemove:structuredDilatingCones:H^ocross_hProperties:H^ocross_hINHeDC(CxR)} yields $H_h^\vb\in\HenigDilatingCones{\c\times\R_\geq}$. Therefore, \Cref{th:PEff:PEffHenigEqui} implies $\x\in\PEffHenig{\fvb}{\calD}{\c\times\R_\geq}$. Since $\x\in\s$, this proves the inclusion and hence the equality.
	\end{proof}

	\begin{Bemerkung}
		A weaker alternative to the global exactness condition in \Cref{th:vectorialPenalisation:PEff:solutionSetRelations:A3Exactness} for achieving the decompositions \eqref{gl:vectorialPenalisation:PEff:solutionSetRelations:BeziehungenPEffs:A3:DI=} and \eqref{gl:vectorialPenalisation:PEff:solutionSetRelations:BeziehungenPEffs:A3:D=} can be formulated in the spirit of \eqref{ass:vectorialPenalisation:PEff:penalisationFunctions:A1} and \eqref{ass:vectorialPenalisation:PEff:penalisationFunctions:A2}, but it has to be local at the boundary points under consideration. Assume that \eqref{ass:vectorialPenalisation:PEff:penalisationFunctions:A1} and \eqref{ass:vectorialPenalisation:PEff:penalisationFunctions:A2} hold. Then boundary points of $\s$ have the same penalisation level, exterior points have strictly larger penalisation values, and interior points have strictly smaller penalisation values. The interior points are therefore the only feasible points with a smaller penalisation value, but exterior points can still destroy proper efficiency of a boundary point for the penalised problem. Their larger penalisation value may be offset by a sufficiently strong improvement in $f$. Boundary competitors are already handled by $\x\in\WEff{f}{\s}{H}$, since their penalisation values agree with that of $\x$.

		Similarly to \Cref{th:vectorialPenalisation:PEff:solutionSetRelations:A3Exactness}, a natural condition excluding this obstruction is to assume that for every $\x\in\boundary[\meinUrbildraum]{\s}\cap\PEffHenig{f}{\s}{\c}$ there exist $H\in\HenigDilatingCones{\c}$ and $h\in\interior[\meinBildraum]{H}$ such that $\x\in\WEff{f}{\s}{H}$ and
		\begin{equation}
			f(\x)-f(y)\notin\interior[\meinBildraum]{H}-(\nu(\x)-\nu(y))h
			\quad\text{for all }y\in\calD\setminus\boundary[\meinUrbildraum]{\s}.
			\label{gl:vectorialPenalisation:PEff:solutionSetRelations:localExactBoundaryCondition}
		\end{equation}
		Equivalently, $\x\in\WEff{\fvb}{\calD}{H_h^\vb}$. Hence this condition says that the decrease of $\nu$ on interior feasible points is not strong enough to compensate for the missing strict improvement in the original objective value, while the increase of $\nu$ outside $\s$ is large enough to prevent exterior points from dominating $\x$ in the penalised ordering.

		Under \eqref{ass:vectorialPenalisation:PEff:penalisationFunctions:A1}, \eqref{ass:vectorialPenalisation:PEff:penalisationFunctions:A2}, the local-global type condition \eqref{gl:vectorialPenalisation:PEff:solutionSetRelations:BeziehungenPEffs:A3:PEffS/PEffDForm}, and the condition \eqref{gl:vectorialPenalisation:PEff:solutionSetRelations:localExactBoundaryCondition}, we get the reverse inclusion in \eqref{gl:vectorialPenalisation:PEff:solutionSetRelations:BeziehungenPEffs:A1:DIsub}. Indeed, interior points are handled by \Cref{th:MOP:Eff:PEff:LocalSolutionsInIntAreGlobalIfConvex}, while boundary points are included in $\boundary[\meinUrbildraum]{\s}\cap\PEffHenig{\fvb}{\calD}{\c\times\R_\geq}$ by \eqref{gl:vectorialPenalisation:PEff:solutionSetRelations:localExactBoundaryCondition}. This condition is weaker than the one in \Cref{th:vectorialPenalisation:PEff:solutionSetRelations:BeziehungenPEffs}~\ref{item:vectorialPenalisation:PEff:solutionSetRelations:BeziehungenPEffs:A3} or \Cref{th:vectorialPenalisation:PEff:solutionSetRelations:A3Exactness}, but it is less intrinsic because it depends on the objective map $f$ and on the relevant dilating cone $H$.
	\end{Bemerkung}

    \begin{Bemerkung}
		Assuming \eqref{ass:vectorialPenalisation:PEff:problemFormulation:VOPRestrVBVor} with $\calD=\meinUrbildraum$, \Cref{th:vectorialPenalisation:PEff:solutionSetRelations:BeziehungenPEffs}~\ref{item:vectorialPenalisation:PEff:solutionSetRelations:BeziehungenPEffs:A1} shows that the constrained problem \eqref{P:P} can be analysed by splitting the set of Henig-properly efficient points into an interior part and a boundary part. More precisely, Henig-properly efficient points in $\interior[\meinUrbildraum]{\s}$ are described by the unconstrained problem \eqref{P:PX}, whereas the ones in $\boundary[\meinUrbildraum]{\s}$ are captured by the penalised unconstrained problem \eqref{P:vectorialPenalisation:PEff:problemFormulation:PvbX}. Hence, the difficulty caused by the constraint set $\s$ is localised at the boundary of $\s$, while the interior behaves as in the unconstrained case. In particular, the use of two unconstrained problems is not merely a technical reformulation; it reflects the intrinsic geometry of the constrained problem. In this sense, \eqref{ass:vectorialPenalisation:PEff:penalisationFunctions:A1} provides the minimal compatibility needed to preserve the relevant ordering information and to obtain the interior--boundary decomposition at the level of inclusions. This decomposition is conceptually useful for 
		the design of effective algorithms for constrained vector optimisation problems, because it allows one to combine tools for unconstrained vector optimisation with a dedicated treatment of the boundary of the feasible set. Part~\ref{item:vectorialPenalisation:PEff:solutionSetRelations:BeziehungenPEffs:A3} is stronger in that the penalised problem captures the constrained one exactly on $\s$, and under the additional boundary condition from \Cref{th:MOP:Eff:PEff:LocalSolutionsInIntAreGlobalIfConvex} this leads to a full decomposition formula. 
    \end{Bemerkung}

\section{Conclusions}\label{s-con}

In this paper, we established inclusion relations between the Henig-properly efficient solution set of a constrained vector optimisation problem and the corresponding solution sets of two associated unconstrained problems. We also identified sufficient exactness conditions under which these inclusions yield an exact interior--boundary decomposition. These include a global condition and a weaker condition imposed only at relevant boundary points. Our examples show that neither the $A_3$ property together with convexity nor the boundary condition alone guarantees the required equalities.

The results distinguish the roles of the two unconstrained problems: the original unconstrained problem accounts for the interior contribution, whereas the penalised unconstrained problem captures boundary information. Under the additional exactness and boundary hypotheses of the relevant results, the inclusions become equalities and provide a complete interior--boundary representation. To support this comparison, we also constructed and analysed a structured family of Henig dilating cones for the ordering cone $\c\times\R_\geq$. The zero-level sections of all Henig dilating cones for $\c\times\R_\geq$ recover exactly the family $\HenigDilatingCones{\c}$, and such cones exist for $\c$ if and only if they exist for $\c\times\R_\geq$. Although the constructed family need not exhaust the full family for the penalised ordering cone, its section and inclusion properties underpin the comparison of the Henig-properly efficient solution sets and the characterisation of when penalisation may create or destroy Henig-proper efficiency.

These relations provide a basis for future work on necessary and sufficient optimality conditions for constrained Henig-proper efficiency. They also suggest algorithms combining methods for unconstrained vector optimisation with a separate treatment of boundary points. Developing and analysing such methods is a natural direction for future research.


	\printbibliography

\end{document}